\documentclass[a4paper]{amsart}

\usepackage{amssymb}
\usepackage{hyperref}

\usepackage{dsfont}

\theoremstyle{plain}

\newtheorem{thm}{Theorem}[section]
\newtheorem{prop}[thm]{Proposition}
\newtheorem{lem}[thm]{Lemma}
\newtheorem{cor}[thm]{Corollary}

\theoremstyle{remark}
\newtheorem*{rema}{Remark}

\numberwithin{equation}{section}

\begin{document}

	\title[Solutions of Navier--Stokes--Maxwell systems in large energy spaces]{Solutions of Navier--Stokes--Maxwell systems \\ in large energy spaces}

	
	\author{Diogo Ars\'enio}
	\address
	{
	New York University Abu Dhabi \\
	Abu Dhabi \\
	United Arab Emirates
	}
	\email{diogo.arsenio@nyu.edu}


	\author{Isabelle Gallagher}
	\address
	{DMA, \'Ecole normale sup\'erieure, CNRS, PSL Research University, 75005 Paris 
	\\
	and UFR de math\'ematiques, Universit\'e Paris-Diderot, Sorbonne Paris-Cit\'e, 75013 Paris, France.}	\email{gallagher@math.ens.fr}

	
	\date{\today}
	
	\begin{abstract}
		Large weak solutions to Navier--Stokes--Maxwell systems are not known to exist in their corresponding energy space in full generality. Here, we mainly focus on the three-dimensional setting of a classical incompressible Navier--Stokes--Maxwell system and --- in an effort to build solutions in the largest possible functional spaces --- prove that global solutions exist under the assumption that the initial velocity and electromagnetic fields have finite energy, and that the initial electromagnetic field is small in $\dot H^s\left({\mathbb R}^3\right)$ with $s\in \left[\frac 12,\frac 32\right)$. We also apply our method to improve known results in two dimensions by providing uniform estimates as the speed of light tends to infinity.

		The method of proof relies on refined energy estimates and a Gr\"onwall-like argument, along with a new maximal estimate on the heat flow in Besov spaces. The latter parabolic estimate allows us to bypass the use of the so-called Chemin--Lerner spaces altogether, which is crucial and could be of independent interest.
	\end{abstract}

	\maketitle


\section{Introduction and main results}

We study the incompressible Navier--Stokes--Maxwell system with Ohm's law in two and three space-dimensions:
\begin{equation}\label{NSM}
	\left\{
	\begin{aligned}
		&\partial_t u +
		u\cdot\nabla  u - \mu\Delta  u
		  = -\nabla p
		+ j\times B \,  ,  & \operatorname{div} u   = 0 \, ,\\
		&\frac 1c\partial_t E - \nabla\times B = - j\,  , & j
		 = \sigma\left( cE + u \times B \right)\,  ,
		\\
		&\frac 1c\partial_t B + \nabla\times E   = 0\,  , & \operatorname{div} B  = 0\,  ,
	\end{aligned}
	\right.
\end{equation}
where $c>0$ denotes the speed of light, $\mu>0$ is the viscosity of the fluid and~$\sigma>0$ is the electrical conductivity.
In the above system, $t\in\mathbb{R^+}$ and~$x\in\mathbb{R}^d$ (where~$d=2$ or~$3$) are the time and space variables, $u=(u_1,u_2,u_3)=u(t,x)$ stands for the velocity field of the (incompressible) fluid while~$E=(E_1,E_2,E_3)=E(t,x)$ and~$B=(B_1,B_2,B_3)=B(t,x)$ are the electric and magnetic fields respectively. All are three-component vector fields. However, when~$d=2$, it is assumed that $u_3=E_3=B_1=B_2=0$. Finally, the scalar function~$p=p(t,x)$ is the pressure and is also an unknown. Observe, though, that the electric current $j=j(t,x)$ is not an unkown, for it is fully determined by $(u,E,B)$ through Ohm's law.

The Navier--Stokes--Maxwell system~\eqref{NSM} describes the evolution of a plasma (i.e.\ a charged fluid) subject to a self-induced electromagnetic Lorentz force $j\times B$. It is by no means the only available description of such a viscous incompressible plasma. Indeed, other similar models coupling the Navier--Stokes equations with Maxwell's equations through different Ohm's laws include
\begin{equation*}
	\left\{
	\begin{aligned}
		&\partial_t u +
		u\cdot\nabla  u - \mu\Delta  u
		  = -\nabla p
		+ j\times B \,  ,  & \operatorname{div} u   = 0 \, ,
		\\
		& \frac 1c\partial_t E - \nabla\times B = - j\,  ,
		& \operatorname{div}B = 0 \, ,
		\\
		&\frac 1c\partial_t B + \nabla\times E   = 0\,  , & \operatorname{div} E  = 0\,  ,
		\\
		& j = \sigma\left( -\nabla\bar p + cE + u \times B \right)\,  ,
		&\operatorname{div} j = 0 \, ,
	\end{aligned}
	\right.
\end{equation*}
where the electromagnetic pressure $\bar p=\bar p(t,x)$ is also unkown, and
\begin{equation*}
	\left\{
	\begin{aligned}
		& \partial_t u + u\cdot\nabla u - \mu\Delta u
		= -\nabla p + cnE + j\times B \, ,
		& \operatorname{div} u = 0 \, ,
		\\
		&\frac 1c\partial_t E - \nabla\times B = - j\, ,
		& \operatorname{div} B  = 0\, ,
		\\
		&\frac 1c\partial_t B + \nabla\times E   = 0\, ,
		& \operatorname{div}E = n \, ,
		\\
		& j - nu = \sigma\left( -c\nabla n + cE + u \times B \right)\, ,
		&
	\end{aligned}
	\right.
\end{equation*}
where the electric charge density $n=n(t,x)$ is not unknown, for it is determined by Gauss's law $\operatorname{div}E=n$.

The appropriateness of each system depends on the specific physical regime under consideration. However, we believe that the Navier--Stokes--Maxwell system~\eqref{NSM} captures most of the essential mathematical difficulties pertaining to the non-linear coupling of the incompressible Navier--Stokes equations with Maxwell's system, which is hyperbolic. From now on, we are therefore going to focus exclusively on~\eqref{NSM}. Nevertheless, we expect that most results concerning~\eqref{NSM} can be extended, in some form, to the other Navier--Stokes--Maxwell systems.

We refer to~\cite{arseniosaintraymond} for systematic derivations of the above systems from kinetic Vlasov--Maxwell--Boltzmann systems, and to~\cite{Biskamp93, Davidson01} for more details on the physics underlying the behavior of plasmas.

\medskip

Before discussing the contents of this paper let us recall some well-established facts regarding the Cauchy problem for the Navier--Stokes equations (corresponding to the case when~$(E,B) \equiv 0$ in~\eqref{NSM}), in relation with this work.  Formally it is easy to see, by multiplying the Navier--Stokes equations by~$u$ and integrating in space, that
\begin{equation*}
	\frac 12\frac{d}{dt} \left\|u(t)\right\|_{L^2}^2 
	+ \mu\left\|\nabla  u(t)\right\|_{L^2}^2  = 0 \, .
\end{equation*}
Using this property, J.\ Leray was able to prove in~\cite{leray} the global existence of bounded energy solutions to the Navier--Stokes equations in
\begin{equation}\label{energyspace}
	L^\infty ({\mathbb R}^+;L^2) \cap L^2 ({\mathbb R}^+;\dot H^1 ) \, ,
\end{equation}
as soon as the initial data~$u_0$ lies in~$L^2$, such that the following energy inequality is satisfied, for   every $t>0$:
\begin{equation*}
	\frac 12  \left\|u(t)\right\|_{L^2}^2
	+ \mu\int_0^t \left \|\nabla  u(\tau)\right\|_{L^2}^2 \, d\tau
	\leq\frac 12 \left\|u_0\right\|_{L^2}^2 \, .
\end{equation*}
The method of proof relies on solving an approximate system (obtained for instance by a frequency cutoff), in proving global in time a priori estimates on the sequence of approximate solutions thanks to the energy bound, and in taking limits in the approximation parameter. Thanks to the smoothing effect provided by the viscosity, the sequence of approximate solutions converges then to a weak solution of the Navier--Stokes equations. There is, however, a possible defect of compactness in the limiting process which leads to the energy being in the end decreasing while it is conserved for the approximate system.

The uniqueness of bounded-energy solutions is, to this day, only known to hold in two space-dimensions, and is also due to J.\ Leray~\cite{leray2D}.

Uniqueness of solutions in general space-dimensions is known for solutions belonging to some {\it scale-invariant spaces}, namely spaces invariant under the transformation
\begin{equation*}
	u(t,x) \mapsto \lambda u (\lambda^2 t, \lambda x) \, ,\quad \lambda>0 \, ,
\end{equation*}
such as~$L^\infty(\mathbb{R}^+;L^d)$ (see~\cite{furiolilemarieterraneo},~\cite{masmoudi2},~\cite{monniaux}).

In two space dimensions, this implies that the energy spaces appearing in~\eqref{energyspace} are scale-invariant. But this property unfortunately fails in higher dimensions, thus rendering the Navier--Stokes equations supercritical whenever $d\geq 3$.

We shall not recall here the extensive literature on the subject, and we only further refer the interested reader to~\cite{bahouri}, \cite{lemarie} or~\cite{lemarie2}, for instance.

\medskip

Let us return now to the full Navier--Stokes--Maxwell equations~\eqref{NSM}. The associate formal energy conservation law is
\begin{equation*}
	\frac 12\frac{d}{dt}\big(\left\|u\right\|_{L^2}^2
	+ \left\|E\right\|_{L^2}^2 + \left\|B\right\|_{L^2}^2 \big)
	+ \mu\left\|\nabla  u\right\|_{L^2}^2 + \frac 1\sigma \left\|j\right\|_{L^2}^2 = 0 \, .
\end{equation*}
It is therefore natural to expect the existence of weak solutions to~\eqref{NSM} such that
\begin{equation}\label{energyspace2}
	u \in L^\infty ({\mathbb R}^+;L^2) \cap L^2 ({\mathbb R}^+;\dot H^1 )\, ,
	\quad 
	\left(E,B\right) \in L^\infty ({\mathbb R}^+;L^2) \,,
	\quad
	j \in L^2 ({\mathbb R}^+;L^2) \, ,
\end{equation}
satisfying the energy inequality, for almost all $t > 0$,
\begin{equation}\label{energy}
	\begin{aligned}
		\frac 12 \big(\left\|u(t)\right\|_{L^2}^2
		+ \left\|E(t)\right\|_{L^2}^2 + \left\|B(t)\right\|_{L^2}^2 \big)
		& + \int_0^t \big( \mu\left\|\nabla u(\tau)\right\|_{L^2}^2 + \frac 1\sigma \left\|j(\tau)\right\|_{L^2}^2 \big) \, d\tau
		\\
		& \leq
		\frac 12\big(\left\|u_0\right\|_{L^2}^2
		+ \left\|E_0\right\|_{L^2}^2 + \left\|B_0\right\|_{L^2}^2 \big) \, ,
	\end{aligned}
\end{equation}
where $\left(u_0,E_0,B_0\right)\in L^2$ is the initial data. For convenience of notation, we henceforth denote the initial energy by
\begin{equation*}
	\mathcal{E}_0:=\frac 12\big(\left\|u_0\right\|_{L^2}^2
	+ \left\|E_0\right\|_{L^2}^2 + \left\|B_0\right\|_{L^2}^2 \big).
\end{equation*}

\medskip

Compared to the Navier--Stokes equations mentioned above, solving~\eqref{NSM} in the energy space seems very difficult as there is not enough compactness in the magnetic field~$B$ to take limits, after an approximation procedure, in the non-linear term~$j \times B$. Furthermore, as discussed in~\cite[Section~2]{arsenio2}, the classical theory of compensated compactness also fails to provide the weak stability of the product~$E\times B$, thus leaving little hope to establish the weak stability of~\eqref{NSM} in its corresponding energy space with classical methods.

A number of studies have recently addressed this lack of compactness in~\eqref{NSM}. In~\cite{Masmoudi10jmpa}, the equations are successfully solved globally in two space-dimensions, for any (possibly large) initial data
\begin{equation*}
	(u_0,E_0,B_0) \in L^2\times H^s\times H^s \, , \quad \text{with }s>0\, .
\end{equation*}
This result is quite satisfying since it covers a very large class of initial data. It remains unknown, though, whether initial electric and magnetic fields in $L^2\setminus \cup_{s>0}H^s$ give rise to a global solution in general.

The existence of solutions in two dimensions is extended in~\cite{germain} to any sufficiently small initial data in
\begin{equation*}
	L^2\times L^2_\mathrm{log}\times L^2_\mathrm{log} \, ,
\end{equation*}
where the space $L^2_\mathrm{log}$ resembles an $H^s$-space with a logarithmic weight on high frequencies instead of an algebraic weight, so that $\cup_{s>0}H^s\subset L^2_\mathrm{log}\subset L^2$. We refer to~\cite{germain} for a precise definition of such spaces. It is to be emphasized that these solutions fail to be global unless the initial data is sufficiently small.

Note that a slightly weaker two-dimensional result had been previously obtained in~\cite{ibrahim} for small initial data in
\begin{equation*}
	\dot B^0_{2,1}\times L^2_\mathrm{log}\times L^2_\mathrm{log} \, .
\end{equation*}
The definition of Besov spaces is recalled in our appendix.

In three space-dimensions, a global unique solution for sufficiently small initial data in
\begin{equation*}
	\dot B^\frac12_{2,1}\times \dot H^\frac 12 \times \dot H^\frac12 \, ,
\end{equation*}
is constructed in~\cite{ibrahim}. This result is also extended in~\cite{germain} to small initial data in
\begin{equation*}
	\dot H^\frac12 \times \dot H^\frac 12 \times \dot H^\frac12 \, .
\end{equation*}

In this work, we aim at extending the preceding three-dimensional results for small initial data to larger functional settings, ultimately reaching subsets of~$L^2\times\ L^2\times L^2$ which are as large as possible and eliminating some restrictions on the size of the initial data. Thus, our first result (see Theorem~\ref{mainthm} below) asserts the existence of weak solutions to~\eqref{NSM} in three dimensions provided the initial data has finite energy~$\mathcal{E}_0<\infty$ and the initial electromagnetic field~$(E_0,B_0)$ alone is small in~$\dot H^\frac 12 \times \dot H^\frac 12$. Note that there is no hope of attaining uniqueness of solutions in this setting since, by choosing~$(E_0,B_0)=0$, it would imply the general uniqueness of solutions to the three-dimensional Navier--Stokes equations.

As a byproduct of our three-dimensional methods, we are also able to revisit (see Theorem~\ref{mainthm2d} below) the two-dimensional existence result from~\cite{Masmoudi10jmpa} by refining its estimates so that they remain uniform in the asymptotic regime~$c\to\infty$. This further allows us to derive the two-dimensional magneto-hydrodynamic system with full rigor in Corollary~\ref{maincor}. Note that the asymptotics as~$c \to \infty$ of global finite energy solutions, provided they exist, has been previously studied in~\cite{arsenio2} in two and three space-dimensions.

\subsection{The three-dimensional result}

We first establish that global existence of solutions to the three-dimensional system~\eqref{NSM} holds whenever the initial datum~$(u_0,E_0,B_0)$ is chosen in the natural energy space~$L^2$, while the electromagnetic field~$(E_0,B_0)$ alone lies in~$\dot H^s$, for some given~$s\in \left[\frac 12,\frac 32\right)$, and is sufficiently small when compared to some non-linear function of the initial energy~$\mathcal{E}_0$. The precise formulation of this result is contained in the following theorem.

\begin{thm}\label{mainthm}
	Let~$s$ be any real number in~$\left[\frac 12,\frac 32\right)$. There is a constant~$C_*>0$ such that, if the initial data~$\left(u_0,E_0,B_0\right)$, with~$\operatorname{div}u_0=\operatorname{div}B_0=0$, belongs to~$\left( L^2 \times (H^s)^2\right) ({\mathbb R}^3)$
	with
	\begin{equation}\label{smalldataHs}
		\|(E_0,B_0)\|_{\dot H^s}
		C_* \mathcal{E}_0^{s-\frac 12}
		e^{C_* \mathcal{E}_0}
		\leq 1\, ,
	\end{equation}
	then there is a global weak solution~$\left(u ,E ,B \right)$ to the three-dimensional Navier--Stokes--Maxwell system~\eqref{NSM} satisfying the energy inequality~\eqref{energy} and enjoying the additional regularity
	\begin{equation}\label{estimateHsEB1}
		\begin{aligned}
			E,B & \in L^\infty(\mathbb{R}^+;\dot H^s)
			\\
			E & \in L^2(\mathbb{R}^+;\dot H^s)
			\\
			u & \in L^1(\mathbb{R}^+;\dot B^\frac 32_{2,1})+L^2(\mathbb{R}^+;\dot B^\frac 32_{2,1})\,.
		\end{aligned}
	\end{equation}
\end{thm}

A preliminary strategy of proof of Theorem~\ref{mainthm} is presented in Section~\ref{strategy}. The actual proof of the theorem is then contained in Sections~\ref{proof1}, \ref{proof2} and \ref{proof3}.

\begin{rema}
	A careful reading of the proof of Theorem~\ref{mainthm} shows that the constant~$C_*>0$ can be chosen independently of the speed of light~$c$ provided~\eqref{smalldataHs} is replaced by
	\begin{equation*}
		\|(E_0,B_0)\|_{\dot H^s}
		C_*\mathcal{E}_0^{s-\frac 12}
		\exp\left(
		C_* \left(c\left(\mathcal{E}_0^\frac 12+\mathcal{E}_0\right)+\mathcal{E}_0\right)
		\right)
		\leq 1\, .
	\end{equation*}
\end{rema}

\subsection{The two-dimensional result}

Our main result in two dimensions comes as a byproduct of the methods developed for the proof of Theorem~\ref{mainthm}. It establishes the existence of weak solutions to~\eqref{NSM} without any restriction on the size of the initial data and is a refinement of the global well-posedness result established in~\cite{Masmoudi10jmpa}.

\begin{thm}\label{mainthm2d}
	Let~$s$ be any real number in~$\left(0,1\right)$ and consider any initial data
	\begin{equation}\label{smalldataHs2d}
		\left(u_0,E_0,B_0\right)\in\left( L^2 \times (H^s)^2\right) ({\mathbb R}^2)\, ,
	\end{equation}
	such that~$\operatorname{div}u_0=\operatorname{div}B_0=0$. Then there is a global weak solution~$\left(u ,E ,B \right)$ to the two-dimensional Navier--Stokes--Maxwell system~\eqref{NSM} satisfying the energy inequality~\eqref{energy} and enjoying the regularity
	\begin{equation}\label{estimateHsEB12d}
		\begin{aligned}
			E,B & \in L_{\mathrm{loc}}^\infty(\mathbb{R}^+;\dot H^s)
			\\
			u & \in L_{\mathrm{loc}}^2(\mathbb{R}^+;L^\infty)\,.
		\end{aligned}
	\end{equation}
	In particular, there exists a constant~$C_*>0$ (which is independent of the speed of light~$c$), such that
	\begin{equation}\label{globalindep}
		\begin{aligned}
			\mathcal{E}_0 & \left(\|E(t)\|_{\dot H^s}^2+\|B(t)\|_{\dot H^s}^2\right)
			\\
			& \leq
			\left(e+\mathcal{E}_0\left(\|E_0\|_{\dot H^s}^2+\|B_0\|_{\dot H^s}^2\right)+\frac{t}{1+\mathcal{E}_0+\mathcal{E}_0^2}\right)^{C_* 2^{C_*(\mathcal{E}_0+\mathcal{E}_0^2)}}
			\, ,
		\end{aligned}
	\end{equation}
	for every~$t>0$.
\end{thm}

The justification of Theorem~\ref{mainthm2d} follows a strategy which is similar to the one for Theorem~\ref{mainthm}. The proof of Theorem~\ref{mainthm2d} is contained in Section~\ref{2dcase}.

\begin{rema}
	When compared with the main result from~\cite{Masmoudi10jmpa}, the above theorem has the advantage of providing a control of the velocity~$u$ in~$L_{\mathrm{loc}}^2(\mathbb{R}^+;L^\infty)$ rather than~$L_{\mathrm{loc}}^1(\mathbb{R}^+;L^\infty)$, as performed in~\cite{Masmoudi10jmpa}. This temporal improvement is the crucial technical refinement allowing us to establish the global bound~\eqref{globalindep} uniformly as the speed of light tends to infinity.
\end{rema}

\begin{rema}
	It will be clear from the proof of Theorem~\ref{mainthm2d} in Section~\ref{2dcase} that the velocity field satisfies the uniform bound
	\begin{equation}\label{globalindepu}
		\|u\|_{L^2([0,t];L^\infty)}^2\leq C
		\left(\mathcal{E}_0+\mathcal{E}_0^2\right)
		\log\left(e+t+\frac{\|B\|_{L^\infty([0,t];\dot H^s)}^2}
		{1+\mathcal{E}_0}\right)\, ,
	\end{equation}
	where~$C>0$ is a constant independent of the speed of light~$c$. In particular, by combining~\eqref{globalindep} and~\eqref{globalindepu}, it is readily seen that the bound~$u \in L_{\mathrm{loc}}^2(\mathbb{R}^+;L^\infty)$ is uniform in~$c$.
\end{rema}

The fact that the estimate~\eqref{globalindep} is independent of the speed of light~$c$ allows us to study the regime $c\to\infty$ and obtain a rigorous derivation of the magneto-hydrodynamic system under rather extensive generality. This is the content of the corollary below and constitutes a rather drastic improvement of the two-dimensional result from~\cite{arsenio2} for the same system~\eqref{NSM} (see Proposition~4.1 therein).

\begin{cor}\label{maincor}
	Let~$s\in(0,1)$ be fixed. For each~$c>0$, consider~$(u^c,E^c,B^c)$ the global and finite energy weak solution of~\eqref{NSM} given by Theorem~{\rm\ref{mainthm2d}} for some uniformly bounded initial data
	\begin{equation*}
		\left(u_0^c,E_0^c,B_0^c\right)\in\left( L^2 \times (H^s)^2\right) ({\mathbb R}^2)\, ,
	\end{equation*}
	such that~$\operatorname{div}u_0^c=\operatorname{div}B_0^c=0$. We suppose that the initial data converges weakly in~$L^2 \times (H^s)^2$, as~$c\to\infty$, towards some
	\begin{equation*}
		\left(u_0,E_0,B_0\right)\in\left( L^2 \times (H^s)^2\right) ({\mathbb R}^2)\, ,
	\end{equation*}
	such that~$\operatorname{div}u_0=\operatorname{div}B_0=0$.
	Then, as~$c\to\infty$, up to extraction of a subsequence, $(u^c,B^c)$ converges weakly to a global and finite energy weak solution~$(u,B)$ of the magneto-hydrodynamic system
	\begin{equation}\label{MHD}
		\left\{
		\begin{aligned}
			&\partial_t u +
			u\cdot\nabla  u - \mu\Delta  u
			  = -\nabla p
			+ (\nabla\times B)\times B \,  ,  & \operatorname{div} u   = 0 \, ,\\
			& \partial_t B -\frac 1\sigma\Delta B   = \nabla\times(u\times B)\,  , & \operatorname{div} B  = 0\,  ,
		\end{aligned}
		\right.
	\end{equation}
	with initial data~$u_{|t=0}=u_0\in L^2$ and~$B_{|t=0}=B_0\in H^s$.
\end{cor}

\begin{proof}
	Using Ohm's law to substitute~$cE^c$ in the Faraday equation in~\eqref{NSM}, we see that we need to pass to the limit in the equivalent system
	\begin{equation}\label{NSMsusbtitute}
		\left\{
		\begin{aligned}
			&\partial_t u^c +
			u^c\cdot\nabla  u^c - \mu\Delta  u^c
			  = -\nabla p^c
			+ j^c\times B^c \,  ,  & \operatorname{div} u^c   = 0 \, ,\\
			&\frac 1c\partial_t E^c - \nabla\times B^c = - j^c\,  , & j^c
			 = \sigma\left( cE^c + u^c \times B^c \right)\,  ,
			\\
			&\partial_t B^c + \frac 1\sigma\nabla\times j^c   = \nabla\times(u^c\times B^c)\,  , & \operatorname{div} B^c  = 0\,  .
		\end{aligned}
		\right.
	\end{equation}
	To this end, note that, according to the energy inequality~\eqref{energy}, we have uniform global bounds on the weak solutions in
	\begin{equation*}
		u^c\in L^\infty_t L^2\cap L^2_t \dot H^1\, ,\quad (E^c,B^c)\in L^\infty_t  L^2\, ,\quad j^c\in L^2_t L^2\,,
	\end{equation*}
	where we denote for simplicity~$ L^p_t X$ for the space~$ L^p(\mathbb{R}^+;X)$.
	Thus, up to extraction of subsequences, we have the weak convergences, as~$c\to\infty$,
	\begin{equation*}
		\begin{aligned}
			(u^c,E^c,B^c) & \stackrel{*}{\rightharpoonup} (u,E,B)\, , & \text{in }& L^\infty_t L^2 \, ,
			\\
			j^c & \rightharpoonup j \, , & \text{in }& L^2_tL^2\,.
		\end{aligned}
	\end{equation*}

	Next, since~$u^c$ is uniformly bounded in
	\begin{equation*}
		L^\infty_t L^2 \cap L^2_t\dot H^1 \subset L^4_t L^4\, ,
	\end{equation*}
	and~$\partial_t u^c$ is bounded in~$L^2_{\mathrm{loc}}H^{-2}$, we deduce, invoking a classical compactness result by Aubin and Lions~\cite{aubin, lions} (see~\cite{simon} for a sharp compactness criterion; here, we advise the use of Corollary~1 from Section~6 in~\cite{simon} for a simple application of such compactness results), that
	\begin{equation*}
		u^c\to u\, , \quad \text{in } L^2_{\mathrm{loc}}L^2\, .
	\end{equation*}
	This strong convergence is sufficient to justify the convergence of the non-linear terms
	\begin{equation*}
		\begin{aligned}
			u^c\cdot\nabla u^c & \rightharpoonup u\cdot\nabla u
			\\
			u^c\times B^c & \rightharpoonup u\times B\, ,
		\end{aligned}
	\end{equation*}
	in the sense of distributions.

	Furthermore, by estimate~\eqref{globalindep}, we also have a uniform bound
	\begin{equation*}
		(E^c,B^c)\in L^\infty_{\mathrm{loc}} \dot H^s\, ,
	\end{equation*}
	and it is readily seen that~$\partial_t B^c$ is bounded in~$L^2_{\mathrm{loc}}H^{-1}$. Therefore, a similar compactness argument yields the strong convergence
	\begin{equation*}
		B^c\to B\, , \quad \text{in } L^2_{\mathrm{loc}}L^2\, ,
	\end{equation*}
	which allows us to deduce the convergence of the remaining non-linear term
	\begin{equation*}
		u^c\times B^c \rightharpoonup u\times B\, ,
	\end{equation*}
	in the sense of distributions.

	All in all, letting~$c\to\infty$ in~\eqref{NSMsusbtitute}, we arrive at the limiting system
	\begin{equation*}
		\left\{
		\begin{aligned}
			&\partial_t u +
			u\cdot\nabla  u - \mu\Delta  u
			  = -\nabla p
			+ j\times B \,  ,  & \operatorname{div} u   = 0 \, ,\\
			& \nabla\times B = j\,  , & 
			\\
			&\partial_t B + \frac 1\sigma\nabla\times j   = \nabla\times(u\times B)\,  , & \operatorname{div} B  = 0\,  .
		\end{aligned}
		\right.
	\end{equation*}
	Finally, eliminating the electric current~$j$ above and recalling the vector identity~$\nabla\times\left(\nabla\times B\right)=\nabla(\operatorname{div}B)-\Delta B$ yields the magneto-hydrodynamic system~\eqref{MHD}.
\end{proof}

\begin{rema}
	The preceding result provides a general derivation of the two-dimensional magneto-hydrodynamic system~\eqref{MHD} for some initial data~$(u_0,B_0)\in L^2\times H^s$, with some fixed~$0<s<1$. Further exploiting the results from~\cite{arsenio2}, it is also possible to derive~\eqref{MHD} for any~$(u_0,B_0)\in L^2\times L^2$.
	
	Indeed, as noted in Proposition~4.1 from~\cite{arsenio2} and in the remark thereafter, the system~\eqref{NSM} will converge towards~\eqref{MHD}, in the regime~$c\to\infty$, as soon as the initial data $\left(u_0^c,E_0^c,B_0^c\right)$ remains uniformly bounded in~$(L^2)^3$ and
	\begin{equation}\label{lightcriterion}
		\lim_{c\to\infty}\frac 1c\|u^c\|_{L^2_{\mathrm{loc}}L^\infty}=0\, .
	\end{equation}
	Therefore, in order to verify the convergence of~\eqref{NSM} towards~\eqref{MHD}, there is no need to impose a uniform bound on the initial electromagnetic field~$(E_0^c,B_0^c)$ in~$(\dot H^s)^2$. Rather, by combining~\eqref{globalindep} and~\eqref{globalindepu}, it is sufficient to consider an initial field uniformly bounded in~$(L^2)^2$ such that~$\|(E_0^c,B_0^c)\|_{\dot H^s}$ may diverge in such a way that~\eqref{lightcriterion} remains valid.
	
	Of course, at this point, by carefully manipulating~\eqref{globalindep} and~\eqref{globalindepu}, it would be possible to extract an explicit rate (as a function of~$c$ and~$\mathcal{E}_0$) of divergence for~$\|(E_0^c,B_0^c)\|_{\dot H^s}$ that would ensure the convergence of~\eqref{NSM}. However, that rate would likely not be optimal and so, we will not bother with an explicit computation of such a rate.	
\end{rema}

\subsection{Strategy of proof}\label{strategy}

The proofs of Theorems~\ref{mainthm} and~\ref{mainthm2d} proceed with a general strategy which is similar to the proof of the Leray theorem concerning the Navier--Stokes equations. Namely, we consider first a solution~$(u_n,E_n,B_n)$, for each~$n\in\mathbb{N}$, of the approximate system
\begin{equation}\label{NSMn}
	\left\{
	\begin{aligned}
		&\partial_t u_n +
		(S_n u_n)\cdot\nabla  u_n -  \mu\Delta  u_n
		  = -\nabla p_n 
		+ j_n \times ( S_n B_n ) \,  ,  \\
		& \frac 1c\partial_t E_n  - \nabla\times B_n  = - j_n \,   ,
		\\
		& \frac 1c\partial_t B_n  + \nabla\times E_n    = 0\,  ,
	\end{aligned}
	\right.
\end{equation}
with
\begin{equation*}
	\operatorname{div} u_n = 0 \, ,
	\quad j_n =   \sigma\left(cE_n + u_n  \times (S_nB_n) \right) \,  ,
	\quad \operatorname{div} B_n   = 0\,  ,
\end{equation*}
where~$S_n$ is defined in Appendix~\ref{LP decomposition} and is a frequency truncation operator to frequencies smaller than~$2^n$. Solving this system globally in time, for any fixed~$n$, in the energy space defined by~\eqref{energy}, for the initial data
$$
(u_n,E_n,B_n)_{|t=0} = S_n (u_0,E_0,B_0) \, ,
$$
is routine matter (see~\cite[Section 12.2]{lemarie2}, for instance).

Furthermore, since the initial data is smooth, it is possible to show that $u_n$, $E_n$ and $B_n$ are also smooth for all times. In particular, the energy estimate is fully justified and there holds
\begin{equation*}
	\begin{aligned}
		\frac 12 \big(\left\|u_n(t)\right\|_{L^2}^2
		+ \left\|E_n(t)\right\|_{L^2}^2 + \left\|B_n(t)\right\|_{L^2}^2 \big)
		& + \int_0^t \big( \mu\left\|\nabla u_n(\tau)\right\|_{L^2}^2 + \frac 1\sigma \left\|j_n(\tau)\right\|_{L^2}^2 \big) \, d\tau
		\\
		& \leq
		\frac 12\left\|S_n (u_0,E_0,B_0)\right\|_{L^2}^2 \leq \mathcal{E}_0 \, ,
	\end{aligned}
\end{equation*}
which constitutes the only available uniform (in~$n$) estimate, so far.

As explained previously, this estimate is not sufficient to take the limit~$n\to \infty$ in the term~$ j_n \times (S_nB_n)$ in order to produce a weak solution of~\eqref{NSM}. However, in the present work, we prove that, under assumption~\eqref{smalldataHs} in three dimensions, or~\eqref{smalldataHs2d} in two dimensions, the approximate electromagnetic field~$(E_n,B_n)(t)$ actually remains in~$\dot H^s$, for some positive~$s$, and satisfies the bounds~\eqref{estimateHsEB1} or~\eqref{estimateHsEB12d} uniformly.

These new uniform estimates provide then enough strong compactness on the sequence of magnetic fields~$B_n$ to justify taking the weak limit~$n\to \infty$ of all terms in~\eqref{NSMn}. This gives then rise, in the limit, to a global weak solution of~\eqref{NSM}, satisfying the energy inequality~\eqref{energy} and the bounds~\eqref{estimateHsEB1} or~\eqref{estimateHsEB12d}. All in all, we see that the justifications of Theorems~\ref{mainthm} and~\ref{mainthm2d} are complete provided~\eqref{estimateHsEB1} and~\eqref{estimateHsEB12d} are respectively established uniformly for the approximate sequence of solutions~$(u_n,E_n,B_n)$.

In the sequel, our goal is therefore to prove that~$(u_n,E_n,B_n)$ belongs to the spaces in~\eqref{estimateHsEB1} or~\eqref{estimateHsEB12d} uniformly. As usual, for the sake of simplicity, keeping in mind that all computations can be fully justified through an approximation procedure, we shall perform all estimates formally on the original system~\eqref{NSM} instead of the approximate system~\eqref{NSMn}.

\medskip

The plan of proof is as follows.

We begin in Section~\ref{wave} by establishing some simple estimates on the damped wave system obtained from the combination of Maxwell's equations with Ohm's law
\begin{equation*}
	\left\{
	\begin{aligned}
		& \frac 1c \partial_t E - \nabla\times B = - j \, ,
		& j =  \sigma ( cE + u \times B ) \,  ,
		\\
		& \frac 1c \partial_t B + \nabla\times E = 0 \, ,
		& \operatorname{div} B  = 0 \, .
	\end{aligned}
	\right.
\end{equation*}
A careful analysis of the damping allows us to improve the dependence of our estimates on the speed of light. We also explain therein what kind of regularity should be expected on the velocity field~$u$ in order to propagate some regularity on the electromagnetic field~$(E,B)$.

Then, in Section~\ref{parabolicregularity} we provide new tools for the study of the Stokes system
\begin{equation}\label{stokesequation}
	\partial_t u  - \mu \Delta  u
	= -\nabla p
	-u\cdot\nabla  u
	+ j\times B \,  ,
	\quad \operatorname{div} u   = 0 \, .
\end{equation}
More precisely, we derive new maximal parabolic estimates showing that solutions to the heat equation can gain up to two derivatives with respect to the source terms in Besov spaces, without resorting to the usual Chemin--Lerner spaces (see Appendix~\ref{LP decomposition} for a definition of such spaces). In fact, we believe that this is an important principle that could be useful beyond its application to the present work.

The proof of Theorem~\ref{mainthm} \emph{per se} is then the subject of Sections~\ref{proof1}, \ref{proof2} and~\ref{proof3}. Section~\ref{proof1} deals with the simpler case~$s=1$ and serves as a primer to the general proof. Then, Section~\ref{proof2} builds upon the estimates from Section~\ref{proof1} to establish Theorem~\ref{mainthm} for the more difficult endpoint case~$s=\frac 12$. Finally, Section~\ref{proof3} uses a simple argument to extend the validity of Theorem~\ref{mainthm} to the whole range~$s\in\left[\frac 12,\frac 32\right)$.

As for the two-dimensional Theorem~\ref{mainthm2d}, its proof is presented in Section~\ref{2dcase} and exploits the machinery originally developed for the three-dimensional Theorem~\ref{mainthm}.

\medskip

Finally, the definitions of Besov and Chemin--Lerner spaces along with some useful properties are recalled in the appendix.

\medskip

Regarding the notation, in the following we denote by~$C$ any generic positive constant depending only on fixed parameters, whose precise value is irrelevant and may change from line to line. When necessary, we will distinguish constants by using appropriate indices. Sometimes, we will also employ the common notation~$A \lesssim B$ to mean~$A \leq C B$, for some generic independent constant~$C>0$.

\section{Estimates on the damped wave flow}\label{wave}

In this section, we control the electromagnetic field~$\left(E,B\right)$ by studying the linear properties of Maxwell's system coupled with Ohm's law:
\begin{equation}\label{Maxwell}
	\begin{cases}
		\begin{aligned}
			\partial_t E - c \nabla \times B + \sigma c^2 E & = -\sigma c u\times B
			\\
			\partial_t B + c \nabla \times E & = 0
			\\
			\operatorname{div} B & =0\, ,
		\end{aligned}
	\end{cases}
\end{equation}
which is contained in~\eqref{NSM}.

As previously mentioned, the Navier--Stokes--Maxwell system~\eqref{NSM} suffers from a dire lack of compactness, which is rooted in the hyperbolic nature of Maxwell's system~\eqref{Maxwell}. Indeed, hyperbolic systems do not offer any regularization properties and, therefore, our only hope at establishing some compactness of the magnetic field resides in propagating some~$\dot H^s$-regularity, for some~$s>0$, through the wave flow.

However, source terms in hyperbolic systems can also be at the origin of the build-up of high frequencies. In consequence, the term~$-\sigma c u\times B$ in \eqref{Maxwell} cannot be handled as an independent source. Rather, it should be viewed as a linear contribution in~$B$ multiplied by some coefficient depending on~$u$. To this end, the velocity~$u$ should belong to a suitable algebra acting on~$\dot H^s$. Recalling the paradifferential product law (see Appendix~\ref{LP decomposition})
\begin{equation}\label{para1}
	\|fg\|_{\dot H^s}
	\lesssim \|f\|_{L^\infty\cap \dot B^\frac d2_{2,\infty}} \|g\|_{\dot H^s}
	\lesssim \|f\|_{\dot B^\frac d2_{2,1}} \|g\|_{\dot H^s}\, ,
\end{equation}
which is valid for all~$s\in (-\frac d2,\frac d2)$, where we used the continuity of the embedding~$\dot B^\frac d2_{2,1}\subset L^\infty\cap \dot B^\frac d2_{2,\infty}$, suggests then that the velocity~$u$ should be controlled in~$\dot B^\frac d2_{2,1}$ in~$x$.

This is quite hopeful, for solutions~$u\in L^\infty L^2\cap L^2\dot H^1$ of the three-dimensional incompressible Navier--Stokes equations (without any electromagnetic components) are known to belong to~$L^1\dot B^\frac 32_{2,1}$ (locally in time). This control can easily be obtained from the estimates on the Stokes flow from Section~\ref{proof1} (see Lemmas~\ref{estimateuvflat} and~\ref{estimateuvsharp}). In fact, our general strategy is based upon replicating such estimates in~$\dot B^\frac 32_{2,1}$ for the full system~\eqref{NSM}.

\begin{rema}
	Observe that the strategy from~\cite{Masmoudi10jmpa} in two dimensions is somewhat similar to ours. Indeed, in that work, the crux of the matter lies in obtaining a control on the fluid velocity~$u$ in~$L^1(L^\infty\cap\dot H^1)$. Recall that, in two dimensions, the space~$\dot H^1$ is continuously embedded into~$\dot B^1_{2,\infty}$ and, therefore, the product rule~\eqref{para1} implies that
	\begin{equation*}
		\|fg\|_{\dot H^s}
		\lesssim \|f\|_{L^\infty\cap\dot B^1_{2,\infty}} \|g\|_{\dot H^s}
		\lesssim \|f\|_{L^\infty\cap\dot H^1} \|g\|_{\dot H^s}\, ,
	\end{equation*}
	for all~$s\in (-1,1)$. In Section~\ref{2dcase}, we also obtain refined estimates on the two-dimensional case by revisiting the well-posedness results from~\cite{Masmoudi10jmpa}.
\end{rema}

The following proposition is a simple linear estimate on~\eqref{Maxwell}. It will allow us to propagate the $\dot H^s$-norm of the electromagnetic field~$(E,B)$ by controlling the fluid velocity~$u$ in~$L^1_t(L^\infty\cap \dot B^\frac d2_{2,\infty})$ or in~$L^2_t(L^\infty\cap \dot B^\frac d2_{2,\infty})$.

\begin{prop}\label{estimatewave}
	Let $s\in (-\frac d2,\frac d2)$. One has the following estimate on the solutions of~\eqref{Maxwell}:
	\begin{equation*}
		F(t)
		\leq F_0
		\exp \Big(C\sigma \int_0^t \big(
		c\|u_1(\tau)\|_{L^\infty\cap\dot B^\frac d2_{2,\infty}}
		+\|u_2(\tau)\|_{L^\infty\cap \dot B^\frac d2_{2,\infty}}^2\big)\, d\tau\Big)\, ,
	\end{equation*}
	for every $t\geq 0$, where we consider any decomposition~$u=u_1+u_2$, with~$u_1\in L^1_t(L^\infty\cap \dot B^\frac d2_{2,\infty})$ and~$u_2\in L^2_t(L^\infty\cap \dot B^\frac d2_{2,\infty})$, $C>0$ is an independent constant and
	\begin{equation*}
		\begin{aligned}
			F(t)&:= \frac 12\Big(\|E(t)\|_{\dot H^s}^2+\|B(t)\|_{\dot H^s}^2+\sigma\int_0^t \|cE(\tau)\|_{\dot H^s}^2\, d\tau\Big)
			\\
			F_0&:= \frac 12\big(\|E_0\|_{\dot H^s}^2+\|B_0\|_{\dot H^s}^2\big)\, .
		\end{aligned}
	\end{equation*}
\end{prop}

\begin{rema}
	It is to be emphasized that the constant $C>0$ above is independent of time. This is quite important since we are aiming at a global existence result. In particular, the fact that the time-integrability of~$u$ can be measured globally in an~$L^1$- or an~$L^2$-norm is of especial significance, for an~$L^2$-integrability requires less decay at infinity. As shown in the proof below, the use of a temporal~$L^2$-norm is permitted by the presence of the term~$\sigma c^2 E$ in~\eqref{Maxwell} which acts as a damping.
\end{rema}

\begin{proof}
	Considering a Littlewood--Paley decomposition of~\eqref{Maxwell} (in the notation of the appendix) and then performing a standard energy estimate results in
	\begin{equation*}
		\begin{aligned}
			\frac 12\left(\|\Delta_k E(t)\|_{L^2}^2+\|\Delta_k B(t)\|_{L^2}^2\right)
			+\sigma \int_0^t \|c\Delta_k E(\tau)\|_{L^2}^2\, d\tau
			\hspace{-70mm}&
			\\
			& =
			\frac 12\left(\|\Delta_k E_0\|_{L^2}^2+\|\Delta_k B_0\|_{L^2}^2\right)
			-\sigma c \int_0^t \!\!\! \int_{\mathbb{R}^3}\Delta_k(u\times B)\cdot\Delta_k E(\tau,x)\, dxd\tau\, .
		\end{aligned}
	\end{equation*}
	Further multiplying the preceding identity by~$2^{2ks}$, using the Cauchy--Schwarz inequality and summing over~$k\in\mathbb{Z}$ yields
	\begin{equation*}
		F(t)
		+\frac \sigma 2 \int_0^t \|cE(\tau)\|_{\dot H^s}^2\, d\tau
		\leq
		F_0
		+\sigma c \int_0^t \|u\times B(\tau)\|_{\dot H^s}\|E(\tau)\|_{\dot H^s}\, d\tau\, .
	\end{equation*}
	Then, we employ the paradifferential product rule~\eqref{para1} to deduce
	\begin{equation*}
		F(t)
		+\frac \sigma 2 \int_0^t \|cE(\tau)\|_{\dot H^s}^2\, d\tau
		\leq
		F_0
		+C \sigma c \int_0^t \|u(\tau)\|_{L^\infty\cap \dot B^\frac d2_{2,\infty}} \|B(\tau)\|_{\dot H^s}\|E(\tau)\|_{\dot H^s}\, d\tau\, .
	\end{equation*}
	
	Next, considering the decomposition~$u=u_1+u_2$, we find
	\begin{equation*}
		\begin{aligned}
			& F(t)
			+\frac \sigma 2 \int_0^t \|cE(\tau)\|_{\dot H^s}^2\, d\tau
			\\
			& \leq
			F_0
			+C\sigma c \int_0^t \|u_1(\tau)\|_{L^\infty\cap \dot B^\frac d2_{2,\infty}}F(\tau)\, d\tau
			\\
			& \quad +\sigma \int_0^t\left( \frac {C^2}{2}\|u_2(\tau)\|_{L^\infty\cap \dot B^\frac d2_{2,\infty}}^2 \|B(\tau)\|_{\dot H^s}^2
			+\frac 1 2\|cE(\tau)\|_{\dot H^s}^2\right)\, d\tau\, ,
		\end{aligned}
	\end{equation*}
	whence
	\begin{equation*}
		F(t)
		\leq
		F_0
		+C\sigma \int_0^t \left(c\|u_1(\tau)\|_{L^\infty\cap \dot B^\frac d2_{2,\infty}}+\|u_2(\tau)\|_{L^\infty\cap \dot B^\frac d2_{2,\infty}}^2\right)
		F(\tau)\, d\tau\, .
	\end{equation*}
	Finally, a classical application of Gr\"onwall's lemma concludes the proof of the proposition.
\end{proof}

\section{Parabolic regularity}\label{parabolicregularity}

In this section we study the forced heat equation
\begin{equation}\label{forcedheat}
	\partial_t w - \Delta w = f \, ,
	\quad w_{|t=0} = w_0\, ,
\end{equation}
as well as the Stokes system~\eqref{stokesequation}, and prove various estimates which will be useful in the sequel. Recall that, using standard semi-group notation, the solution of the above heat equation can be represented as
\begin{equation}\label{duhamel}
	w(t) = e^{t\Delta}w_0 + \int_0^t e^{(t-\tau)\Delta}f(\tau) \, d\tau \, .
\end{equation}

Based on the preceding Duhamel representation formula, it is possible to show (see~\cite[Section~3.4.1]{bahouri}, for instance), employing a Littlewood--Paley decomposition, the following standard parabolic regularity estimate holds in Chemin--Lerner spaces (see the appendix for a definition of these spaces):
\begin{equation}\label{parabolicCL}
	\|w \|_{\widetilde L^m ([0,T];\dot B^{\sigma+2+\frac 2m}_{p,q} )}
	\lesssim
	\|w_0\|_{ \dot B^{\sigma+2}_{p,q}}
	+\|f \|_{\widetilde L^r([0,T];\dot B^{\sigma+\frac 2r}_{p,q})}\, ,
\end{equation}
for any~$\sigma\in\mathbb{R}$ and~$p,q,r,m\in [1,\infty]$, with~$r\leq m$.

If furthermore~$r\leq q\leq m$, note that
\begin{equation*}
	\begin{aligned}
		L^r([0,T];\dot B^{\sigma+\frac 2r}_{p,q})
		& \subset \widetilde L^r([0,T];\dot B^{\sigma+\frac 2r}_{p,q})\, ,
		\\
		\widetilde L^m ([0,T];\dot B^{\sigma+2+\frac 2m}_{p,q} )
		& \subset L^m ([0,T];\dot B^{\sigma+2+\frac 2m}_{p,q} )\, ,
	\end{aligned}
\end{equation*}
so that~\eqref{parabolicCL} implies
\begin{equation}\label{parabolicBL}
	\|w \|_{L^m ([0,T];\dot B^{\sigma+2+\frac 2m}_{p,q} )}
	\lesssim
	\|w_0\|_{ \dot B^{\sigma+2}_{p,q}}
	+\|f \|_{L^r([0,T];\dot B^{\sigma+\frac 2r}_{p,q})}\, .
\end{equation}
This estimate is weaker but has the advantage of involving solely Besov-space valued Lebesgue spaces in time.

Our result below provides a crucial estimate similar to~\eqref{parabolicCL} and~\eqref{parabolicBL} in Besov spaces, which allows us to completely bypass the use of Chemin--Lerner spaces. These latter spaces are notoriously badly behaved in Gr\"onwall-type arguments, which has us believe that the method developed below can potentially be of use in other problems and, as such, is of independent interest. Note that the results discussed in this section are valid in any dimension.

\begin{prop}\label{estimateheat}
	Let~$\sigma\in\mathbb{R}$, $1<r\leq m<\infty$, $p\in[1,\infty]$ and~$1\leq q\leq m$.
	If~$f$ belongs to~$L^r([0,T];\dot B^{\sigma+\frac 2r}_{p,q})$ and~$w_0\equiv 0$, then the solution of the heat equation~\eqref{forcedheat} satisfies
	\begin{equation*}
		\|w \|_{ L^m ([0,T];\dot B^{\sigma+2+\frac 2m}_{p,q} )}
		\lesssim
		\|f \|_{ L^r([0,T];\dot B^{\sigma+\frac 2r}_{p,q})}\, .
	\end{equation*}
\end{prop}

\begin{rema}
	In the preceding estimates, the constants do not depend on $T>0$ so that one can set $T=\infty$, if necessary.
\end{rema}

\begin{rema}
	The significance of Proposition~\ref{estimateheat} resides in that it extends~\eqref{parabolicBL} to values~$1\leq q<r$. Moreover, in that parameter range, estimate~\eqref{parabolicCL} is not stronger nor weaker, it is just different.
\end{rema}

\begin{rema}
	Observe that, by linearity, combining the preceding proposition with \eqref{parabolicCL} yields the estimate
	\begin{equation*}
		\|w \|_{ L^m ([0,T];\dot B^{\sigma+2+\frac 2m}_{p,q} )}
		\lesssim
		\|w_0\|_{ \dot B^{\sigma+2}_{p,q}}
		+\|f \|_{ L^r([0,T];\dot B^{\sigma+\frac 2r}_{p,q})}\, ,
	\end{equation*}
	for all~$\sigma\in\mathbb{R}$, $1<r\leq m<\infty$, $p\in[1,\infty]$ and~$1\leq q\leq m$.
\end{rema}

\begin{proof}
	We consider first the case~$q=1$ and~$m=r$:
	\begin{equation}\label{estimate0}
		\|w \|_{ L^r ([0,T];\dot B^{\sigma+2+\frac 2r}_{p,1} )}
		\lesssim
		\|f \|_{ L^r([0,T];\dot B^{\sigma+\frac 2r}_{p,1})}\, .
	\end{equation}
	The idea is to use a duality argument: it is enough to prove that, if~$g$ is a function in~$L^{r'}([0,T])$ with~$\frac 1r+\frac 1{r'} = 1$, then
	$$
	\int_0^T g(t) \|w(t)\|_{\dot B^{\sigma+2+\frac 2r}_{p,1}} \, dt
	\lesssim
	\|f \|_{L^r([0,T];\dot B^{\sigma+\frac 2r}_{p,1})}
	\|g\|_{L^{r'}([0,T])}\,.
	$$
	To this end, we first write, in the notation of Appendix~\ref{LP decomposition},
	$$
	\int_0^T g(t) \|w(t)\|_{\dot B^{\sigma+2+\frac 2r}_{p,1} } \, dt
	= \sum_{k \in {\mathbb Z}}
	\int_0^T g(t) \|\Delta_k w(t)\|_{L^p}2^{k(\sigma+2+\frac 2r) } \, dt\,.
	$$
	But, employing the representation formula~\eqref{duhamel}, there is an independent constant~$C>0$ such that
	$$
	\|\Delta_k w(t)\|_{L^p} \lesssim \int_0^t e^{-C(t-\tau)2^{2k}}\|\Delta_k f(\tau)\|_{L^p} \, d\tau \, ,
	$$
	so we have
	\begin{equation*}
		\int_0^T g(t) \|w(t)\|_{\dot B^{\sigma+2+\frac 2r}_{p,1} }\, dt
		\lesssim
		\sum_{k \in {\mathbb Z}}
		\int_0^T \!\!\! \int_0^t |g(t)| e^{-C(t-\tau)2^{2k}}\|\Delta_k f(\tau)\|_{L^p}2^{k(\sigma+2+\frac 2r) }\, d\tau dt \,.
	\end{equation*}
	
	Next, we introduce a maximal operator defined by
		\begin{equation*}
			Mg(\tau):=\sup_{\rho>0} \int_0^T \rho\mathds{1}_{\left\{t-\tau\geq 0\right\}}e^{-(t-\tau)\rho}|g(t)| \, dt \,.
		\end{equation*}
	Classical results from harmonic analysis (see~\cite[Theorems 2.1.6 and 2.1.10]{grafakos}) establish that~$M$ is bounded over~$L^a\left([0,T]\right)$, for any~$1<a<\infty$. This is crucial. Indeed, we have now
	$$
	\int_0^T g(t) \|w(t)\|_{\dot B^{\sigma+2+\frac 2r}_{p,1} }\, dt
	\lesssim \sum_{k \in {\mathbb Z}} \int_0^T
	Mg(\tau) \|\Delta_k f(\tau)\|_{L^p}2^{k(\sigma+\frac 2r) }\, d\tau \, ,
	$$
	whence, by definition of~$\dot B^{\sigma+\frac 2r }_{p,1}$,
	$$
	\int_0^T g(t) \|w(t)\|_{\dot B^{\sigma+2+\frac 2r}_{p,1} }\, dt \lesssim \int_0^T Mg(\tau) \|f(\tau)\|_{\dot B^{\sigma+\frac 2r}_{p,1}} \, d\tau \,.
	$$
	We then conclude, by H\"older's inequality, that
	$$
		\begin{aligned}
			\int_0^T Mg(\tau)\|f(\tau)\|_{\dot B^{\sigma+\frac 2r}_{p,1}} \, d\tau 
			& \lesssim \|Mg\|_{L^{r'} ([0,T])}
			\|f\|_{L^r([0,T];\dot B^{\sigma+\frac 2r}_{p,1})} \\
			& \lesssim \|g\|_{L^{r'} ([0,T])}
			\|f\|_{L^r([0,T];\dot B^{\sigma+\frac 2r}_{p,1})} \,,
		\end{aligned}
	$$
	which completes the justification of Proposition~\ref{estimateheat} in the case~$q=1$, $m=r$.
	
	The remaining estimates are obtained by interpolation. More precisely, standard results on the complex method of interpolation (see~\cite[Theorems~5.1.2 and 6.4.5]{bergh}) yield that
	\begin{equation*}
		\left(L^{r_0}([0,T];\dot B^{\sigma_0}_{p_0,q_0}),L^{r_1}([0,T];\dot B^{\sigma_1}_{p_1,q_1})\right)_{[\theta]}
		= L^r([0,T];\dot B^\sigma_{p,q}) \, ,
	\end{equation*}
	for all~$0<\theta<1$, $1\leq r_0,r_1<\infty$, $1\leq p_0,p_1,q_0,q_1\leq\infty$ and~$\sigma_0\neq\sigma_1$, where~$\frac 1r=\frac{1-\theta}{r_0}+\frac{\theta}{r_1}$, $\frac 1p=\frac{1-\theta}{p_0}+\frac{\theta}{p_1}$, $\frac 1q=\frac{1-\theta}{q_0}+\frac{\theta}{q_1}$ and $\sigma=(1-\theta)\sigma_0+\theta\sigma_1$.

	Therefore, interpolating first estimate~\eqref{estimate0} with the estimate
	\begin{equation*}
		\|w \|_{ L^m ([0,T];\dot B^{\sigma+2+\frac 2m}_{p,1} )}
		\lesssim
		\|f \|_{ L^1([0,T];\dot B^{\sigma+2}_{p,1})}\, ,
	\end{equation*}
	directly deduced from~\eqref{parabolicBL} by setting~$r=q=1$ therein, yields that
	\begin{equation*}
		\|w \|_{ L^m ([0,T];\dot B^{\sigma+2+\frac 2m}_{p,1} )}
		\lesssim
		\|f \|_{ L^r([0,T];\dot B^{\sigma+\frac 2r}_{p,1})}\, ,
	\end{equation*}
	for any~$\sigma\in\mathbb{R}$, $1<r\leq m<\infty$ and~$p\in [1,\infty]$.
	
	Finally, further interpolating the latter estimate with the estimate
	\begin{equation*}
		\|w \|_{L^m ([0,T];\dot B^{\sigma+2+\frac 2m}_{p,m} )}
		\lesssim
		\|f \|_{L^r([0,T];\dot B^{\sigma+\frac 2r}_{p,m})}\, ,
	\end{equation*}
	obtained by setting~$q=m$ in \eqref{parabolicBL}, readily concludes the proof of the lemma.
\end{proof}

The next lemma is an \emph{ad hoc} variant of the preceding estimates. It will be useful in the proof of Theorem~\ref{mainthm} in Section~\ref{proof2}, and requires the introduction of the following non-linear quantity:
\begin{equation*}
	\langle f\rangle_X :=
	\inf_{\|f\|_X=\overline f+\widetilde f}
	\Big(c\|\overline f\|_{L^1([0,T])} + \|\widetilde f\|_{L^2([0,T])}^2\Big) \, ,
\end{equation*}
where~$X$ denotes any given Banach space. This definition is inspired by the right-hand side of the estimate from Proposition~\ref{estimatewave}. Observe, however, that $\langle f\rangle_X$ does not define a norm.

By possibly replacing~$\overline f$ and~$\widetilde f$ by~$\overline f\mathds{1}_{\{\overline f\geq 0,\widetilde f\geq 0\}}+\|f\|_X\mathds{1}_{\{\widetilde f<0\}}$ and~$\|f\|_{X}\mathds{1}_{\{\overline f<0\}}+\widetilde f\mathds{1}_{\{\overline f\geq 0,\widetilde f\geq 0\}}$, respectively, one can always assume that~$\overline f$ and~$\widetilde f$ are both non-negative. Indeed, given any decomposition~$\|f\|_X=\overline f+\widetilde f$, we find that
\begin{equation*}
	\begin{aligned}
		\langle f\rangle_X
		& \leq
		\inf_{\substack{\|f\|_X=\overline g+\widetilde g \\ \overline g\geq 0, \widetilde g\geq 0}}
		\Big(c\|\overline g\|_{L^1([0,T])} + \|\widetilde g\|_{L^2([0,T])}^2\Big)
		\\
		& \leq
		c\big \|\overline f\mathds{1}_{\{\overline f\geq 0,\widetilde f\geq 0\}}+\|f\|_X\mathds{1}_{\{\widetilde f<0\}} \big\|_{L^1([0,T])}
		\\
		& \quad + \big\|\|f\|_{X}\mathds{1}_{\{\overline f<0\}}+\widetilde f\mathds{1}_{\{\overline f\geq 0,\widetilde f\geq 0\}}\big\|_{L^2([0,T])}^2
		\\
		& \leq c\|\overline f\|_{L^1([0,T])} + \|\widetilde f\|_{L^2([0,T])}^2\, ,
	\end{aligned}
\end{equation*}
whence, taking the infimum over all such decompositions,
\begin{equation*}
	\langle f\rangle_X =
	\inf_{\substack{\|f\|_X=\overline g+\widetilde g \\ \overline g\geq 0, \widetilde g\geq 0}}
	\Big(c\|\overline g\|_{L^1([0,T])} + \|\widetilde g\|_{L^2([0,T])}^2\Big)\, .
\end{equation*}

Finally, further note that if~$\langle f\rangle_X<\infty$, then~$f$ belongs to~$L^1X+L^2X$. Indeed, it suffices to consider any decomposition~$\|f\|_X=\overline g+\widetilde g$ such that~$\overline g\geq 0$, $\widetilde g\geq 0$ and
\begin{equation*}
	c\|\overline g\|_{L^1([0,T])} + \|\widetilde g\|_{L^2([0,T])}^2< 2\langle f\rangle_X\,.
\end{equation*}
Then, setting~$f_1:=f\mathds{1}_{\{\overline g\geq \widetilde g\}}$ and~$f_2:=f\mathds{1}_{\{\overline g< \widetilde g\}}$ defines a decomposition~$f=f_1+f_2$ such that
\begin{equation}\label{Xdecomposition}
	c\|f_1\|_{L^1([0,T];X)} + \|f_2\|_{L^2([0,T];X)}^2
	\leq
	c\|2\overline g\|_{L^1([0,T])} + \|2\widetilde g\|_{L^2([0,T])}^2
	< 8\langle f\rangle_X\,.
\end{equation}
Reciprocally, if~$f=f_1+f_2$ with~$f_1\in L^1X$ and~$f_2\in L^2X$, then one has the decomposition~$\|f\|_X=\|f\|_X\mathds{1}_{\{\|f_1\|_X\geq\|f_2\|_X\}}+\|f\|_X\mathds{1}_{\{\|f_1\|_X<\|f_2\|_X\}}\in L^1+L^2$, so that~$\langle f\rangle_X<\infty$.

\begin{lem}\label{estimateheatadhoc}
	Let~$\sigma\in\mathbb{R}$ and~$p\in[1,\infty]$.
	If~$f$ lies in~$L^1([0,T];\dot B^{\sigma}_{p,1})+L^2([0,T];\dot B^{\sigma}_{p,1})$ and~$w_0\equiv 0$, then the solution of the heat equation~\eqref{forcedheat} satisfies
	\begin{equation*}
		\langle w\rangle_{\dot B^{\sigma+2}_{p,1}}
		\lesssim
		\langle f\rangle_{\dot B^{\sigma}_{p,1}}
		\, .
	\end{equation*}
\end{lem}

\begin{proof}
	Consider any decomposition~$\|f\|_{\dot B^{\sigma}_{2,1}}=\overline f+\widetilde f$, with~$\overline f,\widetilde f\geq 0$, and set
	\begin{equation*}
		\begin{aligned}
			w_1(t) & := \int_0^t e^{(t-\tau)\Delta} f(\tau)\mathds{1}_{\{\overline f\geq \widetilde f\}}(\tau)\, d\tau
			\\
			w_2(t) & := \int_0^t e^{(t-\tau)\Delta} f(\tau)\mathds{1}_{\{\overline f< \widetilde f\}}(\tau)\, d\tau \,.
		\end{aligned}
	\end{equation*}
	In accordance with the Duhamel representation~\eqref{duhamel} of~$w$, it clearly holds that~$w=w_1+w_2$.
	
	Next, we define
	\begin{equation*}
		\begin{aligned}
			\overline g & := \|w\|_{\dot B^{\sigma+2}_{2,1}}\mathds{1}_{\{\|w_1\|_{\dot B^{\sigma+2}_{2,1}}\geq \|w_2\|_{\dot B^{\sigma+2}_{2,1}}\}}
			\\
			\widetilde g & := \|w\|_{\dot B^{\sigma+2}_{2,1}}\mathds{1}_{\{\|w_1\|_{\dot B^{\sigma+2}_{2,1}}< \|w_2\|_{\dot B^{\sigma+2}_{2,1}}\}}\,,
		\end{aligned}
	\end{equation*}
	so that~$\|w\|_{\dot B^{\sigma+2}_{2,1}}=\overline g+\widetilde g$. Therefore, employing a combination of estimate~\eqref{parabolicBL} with Proposition~\ref{estimateheat}, we obtain that
	\begin{equation*}
		\begin{aligned}
			\langle w\rangle_{\dot B^{\sigma+2}_{p,1}}
			& \leq
			c\|\overline g\|_{L^1([0,T])} + \|\widetilde g\|_{L^2([0,T])}^2
			\\
			& \lesssim
			c\|w_1\|_{L^1([0,T];\dot B^{\sigma+2}_{p,1})} + \|w_2\|_{L^2([0,T];\dot B^{\sigma+2}_{p,1})}^2
			\\
			& \lesssim
			c\|f\mathds{1}_{\{\overline f\geq \widetilde f\}}\|_{L^1([0,T];\dot B^{\sigma}_{p,1})}
			+ \|f\mathds{1}_{\{\overline f< \widetilde f\}}\|_{L^2([0,T];\dot B^{\sigma}_{p,1})}^2
			\\
			& \lesssim
			c\|\overline f\|_{L^1([0,T])}
			+ \|\widetilde f\|_{L^2([0,T])}^2 \,.
		\end{aligned}
	\end{equation*}
	Hence, considering the infimum of the last sum above over all such decompositions concludes the justification of the lemma.
\end{proof}

As a direct consequence of the preceding parabolic regularity estimates, we provide the following application to the two-dimensional incompressible Navier--Stokes equations establishing that Leray solutions satisfy an~$L^2L^\infty$-bound (recall that~$\dot H^1$ fails to embed into~$L^\infty$). Such a bound was originally featured in~\cite{chemin4}, but the proof given below is substantially simpler.

\begin{cor}\label{2dproperty}
	Consider any Leray solution~$u\in L^\infty(\mathbb{R}^+;L^2)\cap L^2(\mathbb{R}^+;\dot H^1)$ to the two-dimensional incompressible Navier--Stokes system, for some divergence-free initial data~$u_0\in L^2$. Then~$u$ belongs to~$L^2(\mathbb{R}^+; L^\infty)$.
\end{cor}

\begin{proof}
	The weak solution is first decomposed uniquely into~$u=u_1+u_2$, where~$u_1$ and~$u_2$ satisfy the respective Stokes systems
	\begin{equation*}
		\left\{
		\begin{aligned}
			\partial_t u_1 - \mu \Delta u_1 &= 0
			\\
			\operatorname{div} u_1 &= 0
			\\
			u_{1|t=0} &  = u_0\, ,
		\end{aligned}
		\right.
		\hspace{15mm}
		\left\{
		\begin{aligned}
			\partial_t u_2 - \mu \Delta u_2 &= -\nabla p - u\cdot\nabla u
			\\
			\operatorname{div} u_2 &= 0
			\\
			u_{2|t=0} &  = 0\, .
		\end{aligned}
		\right.
	\end{equation*}
	
	We estimate~$u_1$ as in~\cite{chemin4}. To be precise, according to the Duhamel representation formula~\eqref{duhamel}, we see that
	\begin{equation*}
		\|u_1\|_{L^2(\mathbb{R}^+;L^\infty)}\leq \|e^{t\Delta}u_0\|_{L^2(\mathbb{R}^+;L^\infty)}
		\lesssim \|u_0\|_{\dot B^{-1}_{\infty,2}} \lesssim \|u_0\|_{L^2},
	\end{equation*}
	where we have used that~$\|e^{t\Delta}u_0\|_{L^2(\mathbb{R}^+;L^\infty)}$ defines an equivalent norm on~$\dot B^{-1}_{\infty,2}$ (see~\cite[Theorem~2.34]{bahouri} for details) and that~$L^2\subset\dot B^{-1}_{\infty,2}$ is a continuous embedding.
	
	As for~$u_2$, we handle it through an application of parabolic regularity estimates, as well. Indeed, denoting the Leray projector onto divergence-free vector fields by~$P:L^2\to L^2$, we deduce from Proposition~\ref{estimateheat} that
	\begin{equation*}
		\|u_2\|_{ L^2 (\mathbb{R}^+;\dot B^{1}_{2,1} )}
		\lesssim
		\|P(u\cdot\nabla u) \|_{ L^2(\mathbb{R}^+;\dot B^{-1}_{2,1})}
		\lesssim
		\|u\otimes u\|_{ L^2(\mathbb{R}^+;\dot B^{0}_{2,1})}\, .
	\end{equation*}
	We emphasize that the classical estimate~\eqref{parabolicCL} would have failed here.
	
	Then, recalling the two-dimensional paradifferential product law (see Appendix~\ref{LP decomposition})
	\begin{equation}\label{para3}
		\|fg\|_{\dot B^{s+t-1}_{2,1}}
		\lesssim \|f\|_{\dot H^s} \|g\|_{\dot H^t}\, ,
	\end{equation}
	which is valid for all~$s,t\in (-1,1)$ with~$s+t>0$, we infer
	\begin{equation*}
		\|u_2\|_{ L^2 (\mathbb{R}^+;\dot B^{1}_{2,1} )}
		\lesssim
		\|u\|_{ L^4(\mathbb{R}^+;\dot H^\frac 12)}^2
		\lesssim
		\|u\|_{ L^\infty(\mathbb{R}^+;L^2)}
		\|u\|_{ L^2(\mathbb{R}^+;\dot H^1)}\, .
	\end{equation*}
	Finally, noticing that~$\dot B^{1}_{2,1}\subset L^\infty$ is a continuous embedding concludes the proof.
\end{proof}

\section{Proof of Theorem~\ref{mainthm} in the case~$s=1$}\label{proof1}

We provide here a justification of Theorem~\ref{mainthm} in the simpler case~$s=1$, which will serve as a primer to the proof of the full case~$s\in[\frac12,\frac32)$. The estimates derived here will be useful in the proof of the full case~$s\in[\frac12,\frac32)$, as well.

Recall that we are considering here a weak solution of the three-dimensional incompressible Navier--Stokes--Maxwell system~\eqref{NSM}, in the functional spaces~\eqref{energyspace2}, satisfying the energy inequality~\eqref{energy}, and that all formal computations can be fully justified by considering smooth solutions of the approximate systems~\eqref{NSMn} instead. The goal of the proof consists in showing the validity of the~$\dot H^1$-bound~\eqref{estimateHsEB1} (where we set~$s=1$) provided~\eqref{smalldataHs} holds initially.

\medskip

Now, we study the Stokes equation~\eqref{stokesequation} and introduce the following decomposition (note that a similar decomposition was already used in~\cite{germain}):
\begin{equation}\label{decomposition}
	u=u_v^\flat+u_v^\sharp+u_e \, ,
\end{equation}
where~$u_v^\flat$ is the solution to the Stokes equation with initial data compactly supported in Fourier space (recall that~$S_0$ is the frequency truncation operator defined in the appendix)
$$
	\left\{
	\begin{aligned}
		\partial_t u_v^\flat - \mu \Delta u_v^\flat &= 0
		\\
		\operatorname{div} u_v^\flat&= 0
		\\
		u_{v|t=0}^\flat& = S_0 u_0
		\, ,
	\end{aligned}
	\right.
$$
and~$u_v^\sharp$ is the ``velocity-part'' of~$u$, with high frequency initial data, solving
$$
\left\{
	\begin{aligned}
		\partial_t u_v^\sharp - \mu \Delta u_v^\sharp &= -\nabla p_v^\sharp- u \cdot \nabla u  \\
		\operatorname{div}  u_v^\sharp&= 0 \\
		 u_{v|t=0}^\sharp&  = (\operatorname{Id} - S_0) u_0\, ,
	\end{aligned}
	\right.
$$
whereas~$u_e$ takes into account the ``electromagnetic-part'' of~$u$, i.e.\ it solves
$$
\left\{
	\begin{aligned}
		\partial_t u_e  - \mu \Delta u_e &= -\nabla p_e  + j \times B \\
		\operatorname{div}  u_e &= 0 \\
		 u_{e|t=0} &  = 0\, .
	\end{aligned}
	\right.
$$

\medskip

Let us start by estimating~$u_v^\flat $.

\begin{lem}\label{estimateuvflat}
	There holds that
	\begin{equation*}
		\|u_v^\flat  \|_{L^2 ({\mathbb R}^+;\dot B^{\frac 32}_{2,1} )}
		\lesssim
		\mathcal{E}_0^\frac 12\, .
	\end{equation*}
\end{lem}

\begin{proof}
	Let~$m>\frac 43$. By~\eqref{parabolicBL}, we find that
	\begin{equation*}
		\|u_v^\flat  \|_{L^m ({\mathbb R}^+;\dot B^{\frac 32}_{2,1} )}
		=
		\|e^{t\Delta}S_0u_0 \|_{L^m ({\mathbb R}^+;\dot B^{\frac 32}_{2,1} )}
		\lesssim
		\|S_0u_0 \|_{\dot B^{\frac 32-\frac 2m}_{2,1}}\, .
	\end{equation*}
	Then, since~$\frac 32-\frac 2m>0$, we further notice that
	\begin{equation*}
		\|S_0u_0 \|_{\dot B^{\frac 32-\frac 2m}_{2,1}}
		=
		\sum_{k\leq 0}
		2^{k(\frac 32-\frac 2m)}
		\|\Delta_k S_0u_0 \|_{L^2}
		\lesssim
		\|u_0\|_{L^2}\, ,
	\end{equation*}
	which concludes the proof choosing~$m=2$.
\end{proof}

Next, we turn to~$u_v^\sharp$.

\begin{lem}\label{estimateuvsharp}
	There holds that
	\begin{equation*}
		\|u_v^\sharp \|_{L^1 ({\mathbb R}^+;\dot B^{\frac 32}_{2,1}   )}
		\lesssim \mathcal{E}_0^\frac 12 + \mathcal{E}_0 \, .
	\end{equation*}
\end{lem}

\begin{proof}
	Let us write the Duhamel representation~\eqref{duhamel} of~$u_v^\sharp$:
	$$
		u_v^\sharp(t) = e^{t\Delta} (\operatorname{Id} - S_0)u_0
		- \int_0^t e^{(t-\tau)\Delta} P \left(u\cdot\nabla u\right)(\tau)\, d\tau \,,
	$$
	where~$P:L^2\to L^2$ is the Leray projector onto divergence-free vector fields. By~\eqref{parabolicBL}, we have
	\begin{equation*}
		\|u_v^\sharp\|_{L^1(\mathbb{R}^+;\dot B^\frac 32_{2,1})}
		\lesssim
		\|(\operatorname{Id} - S_0)u_0\|_{\dot B^{-\frac 12}_{2,1}}
		+
		\|P(u\cdot\nabla u)\|_{L^1(\mathbb{R}^+;\dot B^{-\frac 12}_{2,1})}\, .
	\end{equation*}
	Then, on the one hand, we find
	\begin{equation*}
		\|(\operatorname{Id} - S_0)u_0\|_{\dot B^{-\frac 12}_{2,1}}
		=
		\sum_{k\geq -1} 2^{-\frac k2}
		\|\Delta_k(\operatorname{Id} - S_0)u_0\|_{L^2}
		\lesssim \|u_0\|_{L^2}\, .
	\end{equation*}
	On the other hand, recalling the three-dimensional paradifferential product law (see Appendix~\ref{LP decomposition})
	\begin{equation}\label{para2}
		\|fg\|_{\dot B^{s+t-\frac 32}_{2,1}}
		\lesssim \|f\|_{\dot H^s} \|g\|_{\dot H^t}\, ,
	\end{equation}
	which is valid for all~$s,t\in (-\frac 32,\frac 32)$ with~$s+t>0$, we infer
	\begin{equation*}
		\begin{aligned}
			\|P(u\cdot\nabla u)\|_{L^1(\mathbb{R}^+;\dot B^{-\frac 12}_{2,1})}
			& \leq
			\|u\cdot\nabla u\|_{L^1(\mathbb{R}^+;\dot B^{-\frac 12}_{2,1})}
			\\
			& \lesssim
			\big\|
			\|u\|_{\dot H^1}\|\nabla u\|_{L^2}
			\big\|_{L^1(\mathbb{R}^+)}
			\lesssim \|u\|_{L^2(\mathbb{R}^+;\dot H^1)}^2\, ,
		\end{aligned}
	\end{equation*}
	which concludes the proof.
\end{proof}

We move on now to estimating~$u_e$.

\begin{lem}\label{estimateue}
	There holds that
	\begin{equation*}
		\|u_e \|_{L^2 ([0,T];\dot B^{\frac 32}_{2,1}   )}
		\lesssim
		\big\|\|j\|_{L^2}\|B\|_{\dot H^1}\big\|_{L^2([0,T])}
		\lesssim
		\mathcal{E}_0^\frac 12\|B\|_{L^\infty([0,T];\dot H^1)} \, .
	\end{equation*}
\end{lem}

\begin{proof}
	Let us write the Duhamel representation~\eqref{duhamel} of~$u_e$:
	$$
		u_e(t) = \int_0^t e^{(t-\tau)\Delta} P\left(j\times B\right)(\tau)\, d\tau \,.
	$$
	By Proposition~\ref{estimateheat}, we have
	\begin{equation}\label{crucialestimate}
		\|u_e\|_{L^2([0,T];\dot B^\frac 32_{2,1})}
		\lesssim
		\|P(j\times B)\|_{L^2([0,T];\dot B^{-\frac 12}_{2,1})}
		\lesssim
		\|j\times B\|_{L^2([0,T];\dot B^{-\frac 12}_{2,1})}\, .
	\end{equation}
	Therefore, employing the paradifferential product rule~\eqref{para2} yields
	\begin{equation*}
		\|u_e\|_{L^2([0,T];\dot B^\frac 32_{2,1})}
		\lesssim
		\big\|\|j\|_{L^2}\|B\|_{\dot H^1}\big\|_{L^2([0,T])}
		\lesssim
		\|j\|_{L^2([0,T];L^2)}\|B\|_{L^\infty([0,T];\dot H^1)}\, ,
	\end{equation*}
	which concludes the proof.
\end{proof}

\begin{rema}
	It is to be emphasized that the parabolic regularity estimate~\eqref{parabolicBL} would have failed to establish Lemma~\ref{estimateue}. It is precisely in~\eqref{crucialestimate} above that Proposition~\ref{estimateheat} plays a fundamental role in allowing us to reach a global existence result.
\end{rema}

We are now in a position to conclude the proof of Theorem~\ref{mainthm}. Indeed, combining Proposition~\ref{estimatewave} (for~$s=1$) with Lemmas~\ref{estimateuvflat}, \ref{estimateuvsharp} and~\ref{estimateue}, and recalling that the embedding~$\dot B^\frac 32_{2,1}\subset L^\infty\cap\dot B^\frac 32_{2,\infty}$ is continuous, we arrive at
\begin{equation}\label{Fexpbound}
	F(t)\leq F_0
	\exp\left(
	C\left(c\left(\mathcal{E}_0^\frac 12+\mathcal{E}_0\right)+\mathcal{E}_0
	+\int_0^t\|j(\tau)\|_{L^2}^2F(\tau)\, d\tau\right)
	\right)\, ,
\end{equation}
for some constant~$C>0$ depending only on fixed parameters, where we have used the notation of Proposition~\ref{estimatewave}.

We are going to apply the following Gr\"onwall lemma to the preceding inequality.

\begin{lem}\label{gronwall}
	Let~$y(t)\in C([0,T];{\mathbb R}^+)$, $a(t)\in L^1([0,T];{\mathbb R}^+)$ and~$y_0\in\mathbb{R}$ be such that 
	\begin{equation*}
		y_0\int_0^ta(\tau)\, d\tau < 1\, ,
	\end{equation*}
	and
	\begin{equation*}
		y(t)\leq y_0 \exp\left(\int_0^t a(\tau)y(\tau)\, d\tau\right)\, ,
	\end{equation*}
	for every~$t\in [0,T]$. Then, it holds that
	\begin{equation*}
		y(t)\leq \frac{y_0}{1-y_0\int_0^ta(\tau)\, d\tau}\, ,
	\end{equation*}
	for every~$t\in[0,T]$.
\end{lem}

\begin{proof}
	Set
	\begin{equation*}
		f(t):=\exp\left(-\int_0^t a(\tau)y(\tau)\, d\tau\right)\, ,
	\end{equation*}
	for every~$t\in[0,T]$, so that~$y(t)f(t)\leq y_0$. Then, we compute
	\begin{equation*}
		f'(t)=-a(t)y(t)f(t)\geq -a(t)y_0\, ,
	\end{equation*}
	whence, integrating,
	\begin{equation*}
		f(t) \geq 1 -y_0\int_0^ta(\tau)\, d\tau\, .
	\end{equation*}
	Using again that~$y(t)f(t)\leq y_0$, we deduce
	\begin{equation*}
		y_0 \geq y(t)\left(1 -y_0\int_0^ta(\tau)\, d\tau\right)\, ,
	\end{equation*}
	which concludes the proof.
\end{proof}

Thus, applying Lemma~\ref{gronwall} to inequality~\eqref{Fexpbound}, we deduce that
\begin{equation}\label{postGronwall}
	F(t)\leq
	\frac{F_0\exp\left(
	C\left(c\left(\mathcal{E}_0^\frac 12+\mathcal{E}_0\right)+\mathcal{E}_0\right)
	\right)}
	{1-CF_0\exp\left(
	C\left(c\left(\mathcal{E}_0^\frac 12+\mathcal{E}_0\right)+\mathcal{E}_0\right)
	\right)\int_0^t \|j(\tau)\|_{L^2}^2 \,d\tau}
	\, ,
\end{equation}
as long as
\begin{equation*}
	CF_0\int_0^t \|j(\tau)\|^2 \,d\tau<
	\exp\left(
	-C\left(c\left(\mathcal{E}_0^\frac 12+\mathcal{E}_0\right)+\mathcal{E}_0\right)
	\right)\, .
\end{equation*}
Therefore, by considering any large constant~$C_*
\geq C\max\{1,2\sigma\}$, we finally conclude that if
\begin{equation*}
	C_*F_0\mathcal{E}_0\leq
	\exp\left(
	-C_*\left(c\left(\mathcal{E}_0^\frac 12+\mathcal{E}_0\right)+\mathcal{E}_0\right)
	\right)\, ,
\end{equation*}
then
\begin{equation}\label{hypothesis}
	CF_0\int_0^t \|j(\tau)\|^2 \,d\tau\leq \frac 12
	\exp\left(
	-C\left(c\left(\mathcal{E}_0^\frac 12+\mathcal{E}_0\right)+\mathcal{E}_0\right)
	\right)\, ,
\end{equation}
whence, combining~\eqref{postGronwall} and~\eqref{hypothesis},
\begin{equation*}
	F(t)\mathcal{E}_0\leq 2F_0\mathcal{E}_0 \exp\left(
	C\left(c\left(\mathcal{E}_0^\frac 12+\mathcal{E}_0\right)+\mathcal{E}_0\right)
	\right)\leq \frac 2{C_*}\, ,
\end{equation*}
for every~$t\in\mathbb{R}^+$. The proof of Theorem~\ref{mainthm} for~$s=1$ is now complete. \qed

\section{Proof of Theorem~\ref{mainthm} in the case~$s=\frac 12$}\label{proof2}

We focus now on the proof of the case~$s=\frac 12$. To this end, we consider the exact same decomposition~\eqref{decomposition} of~$u$ as in the previous section. Lemmas~\ref{estimateuvflat} and~\ref{estimateuvsharp} will serve to estimate the components~$u_v^\flat$ and~$u_v^\sharp$ here as well. As for~$u_e$, it will be handled through another estimate, whose starting point consists in using Ohm's law to control the electric current~$j$, rather than using the sole fact that~$j\in L^2L^2$ according to the energy inequality~\eqref{energy}.

More precisely, we have the following result.

\begin{lem}\label{estimateue2}
	There exists an independent constant~$C_0>0$ such that
	\begin{equation*}
		\begin{aligned}
			\langle u_e\rangle_{\dot B^{\frac 32}_{2,1}}
			& \leq
			\frac{(c(\mathcal{E}_0^\frac 12+\mathcal{E}_0)+\|cE\|_{L^2([0,T];\dot H^\frac 12)}^2)\|B\|_{L^\infty([0,T];\dot H^\frac 12)}^2}
			{C_0-\|B\|_{L^\infty([0,T];\dot H^\frac 12)}^2
			\Big(1+\|B\|_{L^\infty([0,T];\dot H^\frac 12)}^2\Big)}
			\\
			& \quad +
			\frac{
			\mathcal{E}_0\|B\|_{L^\infty([0,T];\dot H^\frac 12)}^4}
			{C_0-\|B\|_{L^\infty([0,T];\dot H^\frac 12)}^2
			\Big(1+\|B\|_{L^\infty([0,T];\dot H^\frac 12)}^2\Big)}
			\,,
		\end{aligned}
	\end{equation*}
	provided~$\|B\|_{L^\infty([0,T];\dot H^\frac 12)}^2\Big(1+\|B\|_{L^\infty([0,T];\dot H^\frac 12)}^2\Big)<C_0$.
\end{lem}

\begin{proof}
	An application of Lemma~\ref{estimateheatadhoc} first produces the estimate
	\begin{equation}\label{half 1}
		\langle u_e\rangle_{\dot B^{\frac 32}_{2,1}}
		\lesssim
		\langle j\times B\rangle_{\dot B^{-\frac 12}_{2,1}}
		\, .
	\end{equation}
	For later use, let us consider some decomposition~$\|u_e\|_{\dot B^{\frac 32}_{2,1}}=\overline f+\widetilde f$, with~$\overline f,\widetilde f\geq 0$, such that
	\begin{equation}\label{decomposition f}
		c\|\overline f\|_{L^1([0,T])}+\|\widetilde f\|_{L^2([0,T])}^2 \leq 2 \langle u_e\rangle_{\dot B^{\frac 32}_{2,1}}\,.
	\end{equation}
	Next, using the paradifferential product rule~\eqref{para2} yields that
	\begin{equation*}
		\|j\times B\|_{\dot B^{-\frac 12}_{2,1}}
		\lesssim
		\|j\|_{\dot H^{\frac 12}}\|B\|_{\dot H^\frac 12}
		\, .
	\end{equation*}
	Recall now that~$j$ is characterized by Ohm's law~$j=\sigma(cE+u\times B)$. Hence, by virtue of the paradifferential product law~\eqref{para1},
	there holds
	\begin{equation*}
		\|j\|_{\dot H^\frac 12} \lesssim \|cE\|_{\dot H^\frac 12} + \|u\|_{\dot B^\frac 32_{2,1}}\|B\|_{\dot H^\frac 12}\, ,
	\end{equation*}
	which, when combined with the previous estimate, produces the control
	\begin{equation*}
		\begin{aligned}
			\|j\times B\|_{\dot B^{-\frac 12}_{2,1}}
			& \lesssim
			\|cE\|_{\dot H^\frac 12}\|B\|_{\dot H^\frac 12}
			+
			\|u\|_{\dot B^\frac 32_{2,1}}\|B\|_{\dot H^\frac 12}^2
			\\
			& \lesssim
			\|cE\|_{\dot H^\frac 12}\|B\|_{\dot H^\frac 12}
			+
			(\|u_v^\flat\|_{\dot B^\frac 32_{2,1}}+\widetilde f)\|B\|_{\dot H^\frac 12}^2
			\\
			&\quad +
			(\|u_v^\sharp\|_{\dot B^\frac 32_{2,1}}+\overline f)\|B\|_{\dot H^\frac 12}^2
			\, .
		\end{aligned}
	\end{equation*}
	
	In particular, it follows that~$\|j\times B\|_{\dot B^{-\frac 12}_{2,1}}$ can be decomposed as~$\overline g+\widetilde g$ with
	\begin{equation*}
		\begin{aligned}
			0 & \leq \overline g
			\lesssim
			(\|u_v^\sharp\|_{\dot B^\frac 32_{2,1}}+\overline f)\|B\|_{\dot H^\frac 12}^2
			=:\overline h
			\\
			0 & \leq \widetilde g
			\lesssim
			\|cE\|_{\dot H^\frac 12}\|B\|_{\dot H^\frac 12}
			+
			(\|u_v^\flat\|_{\dot B^\frac 32_{2,1}}+\widetilde f)\|B\|_{\dot H^\frac 12}^2
			=:\widetilde h
			\,.
		\end{aligned}
	\end{equation*}
	Indeed, it suffices to set~$\overline g=\|j\times B\|_{\dot B^{-\frac 12}_{2,1}}\mathds{1}_{\{\overline h\geq\widetilde h\}}$ and~$\widetilde g=\|j\times B\|_{\dot B^{-\frac 12}_{2,1}}\mathds{1}_{\{\overline h<\widetilde h\}}$, for instance. It now holds that
	\begin{equation*}
		\begin{aligned}
			\langle j\times B\rangle_{\dot B^{-\frac 12}_{2,1}}
			& \leq
			c\|\overline g\|_{L^1([0,T])}+\|\widetilde g\|_{L^2([0,T])}^2
			\lesssim
			c\|\overline h\|_{L^1([0,T])}+\|\widetilde h\|_{L^2([0,T])}^2
			\\
			& \lesssim
			c\|u_v^\sharp\|_{L^1(0,T;\dot B^\frac 32_{2,1})}\|B\|_{L^\infty([0,T];\dot H^\frac 12)}^2
			+
			c\|\overline f\|_{L^1([0,T])}\|B\|_{L^\infty([0,T];\dot H^\frac 12)}^2
			\\
			& \quad +\|cE\|_{L^2([0,T];\dot H^\frac 12)}^2\|B\|_{L^\infty([0,T];\dot H^\frac 12)}^2
			\\
			& \quad
			+\|u_v^\flat\|_{L^2([0,T];\dot B^\frac 32_{2,1})}^2\|B\|_{L^\infty([0,T];\dot H^\frac 12)}^4
			+
			\|\widetilde f\|_{L^2([0,T])}^2\|B\|_{L^\infty([0,T];\dot H^\frac 12)}^4 \,,
		\end{aligned}
	\end{equation*}
	which implies, recalling~\eqref{decomposition f} and invoking Lemmas~\ref{estimateuvflat} and~\ref{estimateuvsharp}, that
	\begin{equation}\label{half 2}
		\begin{aligned}
			\langle j\times B\rangle_{\dot B^{-\frac 12}_{2,1}}
			& \lesssim
			c(\mathcal{E}_0^\frac 12+\mathcal{E}_0)\|B\|_{L^\infty([0,T];\dot H^\frac 12)}^2
			+
			\mathcal{E}_0\|B\|_{L^\infty([0,T];\dot H^\frac 12)}^4
			\\
			& \quad +\|cE\|_{L^2([0,T];\dot H^\frac 12)}^2\|B\|_{L^\infty([0,T];\dot H^\frac 12)}^2
			\\
			& \quad
			+
			\langle u_e\rangle_{\dot B^{\frac 32}_{2,1}}\|B\|_{L^\infty([0,T];\dot H^\frac 12)}^2
			\left(1+\|B\|_{L^\infty([0,T];\dot H^\frac 12)}^2\right) \,.
		\end{aligned}
	\end{equation}

	Therefore, combining~\eqref{half 1} with~\eqref{half 2}, we finally find that there exists an independent constant~$C_0>0$ such that
	\begin{equation*}
		\begin{aligned}
			\Big(
			C_0-\|B\|_{L^\infty([0,T];\dot H^\frac 12)}^2
			\Big(1+\|B\|_{L^\infty([0,T];\dot H^\frac 12)}^2\Big)
			\Big)
			\langle u_e\rangle_{\dot B^{\frac 32}_{2,1}}
			\hspace{-70mm}&
			\\
			& \leq
			(c(\mathcal{E}_0^\frac 12+\mathcal{E}_0)+\|cE\|_{L^2([0,T];\dot H^\frac 12)}^2)\|B\|_{L^\infty([0,T];\dot H^\frac 12)}^2
			+
			\mathcal{E}_0\|B\|_{L^\infty([0,T];\dot H^\frac 12)}^4
			\,,
		\end{aligned}
	\end{equation*}
	as long as~$\|B\|_{L^\infty([0,T];\dot H^\frac 12)}^2\Big(1+\|B\|_{L^\infty([0,T];\dot H^\frac 12)}^2\Big)<C_0$. The proof of the lemma is now complete.
\end{proof}

\begin{rema}
	As before, we emphasize here that the new parabolic estimates from Section~\ref{parabolicregularity} are critical for the proof of the preceding lemma. In particular, note that the bound~\eqref{half 1} cannot be justified solely with the classical estimate~\eqref{parabolicBL} and requires the use of Proposition~\ref{estimateheat} (through an application of Lemma~\ref{estimateheatadhoc}).
\end{rema}

We proceed now to the conclusion of the proof of Theorem~\ref{mainthm}. To this end, we first decompose the velocity field~$u$ as
\begin{equation*}
	u=(u_v^\sharp+u_e^1)+(u_v^\flat+u_e^2) \in L^1([0,T];\dot B^\frac 32_{2,1})+L^2([0,T];\dot B^\frac 32_{2,1})\, ,
\end{equation*}
where we use a decomposition~$u_e=u_e^1+u_e^2$ with the property~\eqref{Xdecomposition}, that is
\begin{equation*}
	c\|u_e^1\|_{L^1([0,T];\dot B^\frac 32_{2,1})} + \|u_e^2\|_{L^2([0,T];\dot B^\frac 32_{2,1})}^2
	\lesssim \langle u_e\rangle_{\dot B^\frac 32_{2,1}}\,.
\end{equation*}
Then, combining Proposition~\ref{estimatewave} (for~$s=\frac 12$) with Lemmas~\ref{estimateuvflat}, \ref{estimateuvsharp} and \ref{estimateue2}, we deduce the existence of a small constant~$C_0>0$ and a large constant~$C_1>0$ such that, as long as~$G(t)+G(t)^2<C_0$,
\begin{equation*}
	G(t)
	\leq G_0
	\exp\Big(
	C_1\big(c\left(\mathcal{E}_0^\frac 12+\mathcal{E}_0\right)+\mathcal{E}_0
	+\frac
	{\big(c(\mathcal{E}_0^\frac 12+\mathcal{E}_0)+G(t)\big)G(t)
	+
	\mathcal{E}_0G(t)^2}
	{C_0-G(t)-G(t)^2}
	\big)
	\Big)\, ,
\end{equation*}
where we have introduced the notation
\begin{equation*}
	\begin{aligned}
		G(t)&:=
		\sup_{r\in [0,t]}
		\frac 12\left(\|E(r)\|_{\dot H^s}^2+\|B(r)\|_{\dot H^s}^2+\sigma\int_0^r \|cE(\tau)\|_{\dot H^s}^2\, d\tau\right)
		\\
		G_0&:= \frac 12\left(\|E_0\|_{\dot H^s}^2+\|B_0\|_{\dot H^s}^2\right)\, .
	\end{aligned}
\end{equation*}

Recall that all unkowns are assumed to be smooth, for all estimates are to be performed on the regularized system~\eqref{NSMn}. In particular, $G(t)$ is assumed here to be continuous. Note, also, that it is non-decreasing.

Now, let us suppose there exists a finite time~$t_*>0$ such that
$$
G(t_*)+G(t_*)^2=\frac{C_0}{2} \quad \mbox{i.e.}  \quad G(t_*)=\frac{\sqrt{1+2C_0}-1}{2} \, \cdotp 
$$
It follows that
\begin{equation}\label{estimate G}
	\begin{aligned}
		G(t_*)
		& \leq G_0
		\exp\Big(
		C_1
		\frac
		{\big(c(\mathcal{E}_0^\frac 12+\mathcal{E}_0)+\mathcal{E}_0\big)C_0
		-\mathcal{E}_0G(t_*)
		+\big(1-c(\mathcal{E}_0^\frac 12+\mathcal{E}_0)\big)G(t_*)^2}
		{C_0-G(t_*)-G(t_*)^2}
		\Big)
		\\
		& \leq G_0
		\exp\Big(
		2C_1\big(c(\mathcal{E}_0^\frac 12+\mathcal{E}_0)+\mathcal{E}_0\big)
		+\frac{C_0C_1}{2}
		\Big)\, .
	\end{aligned}
\end{equation}
Thus, we reach a contradiction whenever the initial datum is assumed to satisfy that
\begin{equation}\label{initial}
	G_0
	\exp\left(
	2C_1\left(c(\mathcal{E}_0^\frac 12+\mathcal{E}_0)+\mathcal{E}_0\right)
	+\frac{C_0C_1}{2}
	\right)
	<\frac{\sqrt{1+2C_0}-1}{2}\, \cdotp
\end{equation}
In other words, we conclude that, whenever~\eqref{initial} holds, one has
\begin{equation*}
	G(t)<\frac{\sqrt{1+2C_0}-1}{2}\,,
\end{equation*}
for every~$t\geq 0$.

Therefore, we finally conclude that there exists some possibly large constant~$C_*>0$ such that if
\begin{equation*}
	C_*G_0\leq
	\exp\left(
	-C_*\left(c\left(\mathcal{E}_0^\frac 12+\mathcal{E}_0\right)+\mathcal{E}_0\right)
	\right)\, ,
\end{equation*}
then, repeating estimate~\eqref{estimate G} for all~$t\geq 0$,
\begin{equation*}
	G(t)\leq C_*G_0
	\exp\left(
	C_*\left(c\left(\mathcal{E}_0^\frac 12+\mathcal{E}_0\right)+\mathcal{E}_0\right)
	\right)\leq 1\, ,
\end{equation*}
for every~$t\in\mathbb{R}^+$. The proof of Theorem~\ref{mainthm} for~$s=\frac 12$ is now complete. \qed

\section{Proof of Theorem~\ref{mainthm} in the case~$s\in\left[\frac 12,\frac 32\right)$}\label{proof3}

Here, we extend our existence result for~$s=\frac 12$, established in the preceding section, to the whole range of parameters~$s\in \left[\frac 12,\frac 32\right)$. This is simple. Indeed, fixing the value of the parameter~$s\in\left(\frac 12,\frac 32\right)$, by virtue of the interpolation inequality
\begin{equation*}
	\begin{aligned}
		\frac 12\left(\|E_0\|_{\dot H^\frac 12}^2+\|B_0\|_{\dot H^\frac 12}^2\right)
		& \leq
		\left(\frac 12\left(\|E_0\|_{\dot H^s}^2+\|B_0\|_{\dot H^s}^2\right)\right)^{\frac 1{2s}}
		\\
		& \quad \times
		\left(\frac 12\left(\|E_0\|_{L^2}^2+\|B_0\|_{L^2}^2\right)\right)^{1-\frac 1{2s}}
		\\
		& \leq
		\left(\frac 12\left(\|E_0\|_{\dot H^s}^2+\|B_0\|_{\dot H^s}^2\right)\right)^{\frac 1{2s}}
		\mathcal{E}_0^{1-\frac 1{2s}}\, ,
	\end{aligned}
\end{equation*}
we see that, for any given constant~$C_*>0$, it holds
\begin{equation*}
	C_*\frac 12\left(\|E_0\|_{\dot H^\frac 12}^2+\|B_0\|_{\dot H^\frac 12}^2\right)\leq
	\exp\left(
	-C_*\left(c\left(\mathcal{E}_0^\frac 12+\mathcal{E}_0\right)+\mathcal{E}_0\right)
	\right)\, ,
\end{equation*}
as soon as
\begin{equation}\label{full1}
	C_*^{2s}\left(\frac 12\left(\|E_0\|_{\dot H^s}^2+\|B_0\|_{\dot H^s}^2\right)\right)
	\mathcal{E}_0^{2s-1}
	\leq \exp\left(
	-2sC_* \left(c\left(\mathcal{E}_0^\frac 12+\mathcal{E}_0\right)+\mathcal{E}_0\right)
	\right)\, .
\end{equation}

Therefore, assuming that the initial data satisfies~\eqref{full1} for some sufficiently large constant~$C_*>0$, we deduce from Theorem~\ref{mainthm} for~$s=\frac 12$ that there exists a global weak solution of the Navier--Stokes--Maxwell system~\eqref{NSM} such that~$E,B\in L^\infty(\mathbb{R}^+;\dot H^\frac 12)$, $E\in L^2(\mathbb{R}^+;\dot H^\frac 12)$ and~$u\in L^1(\mathbb{R}^+;\dot B^\frac 32_{2,1})+L^2(\mathbb{R}^+;\dot B^\frac 32_{2,1})$.

Finally, a direct application of Proposition~\ref{estimatewave} shows that the electromagnetic field actually enjoys the regularity~$E,B\in L^\infty(\mathbb{R}^+;\dot H^s)$ and~$E\in L^2(\mathbb{R}^+;\dot H^s)$, which concludes the proof of the whole theorem. \qed

\section{Proof of Theorem~\ref{mainthm2d}}\label{2dcase}

We provide here the proof of Theorem~\ref{mainthm2d} based on the proof of Theorem~1.1 from~\cite{Masmoudi10jmpa}.

We are now considering a weak solution of the two-dimensional incompressible Navier--Stokes--Maxwell system~\eqref{NSM}, in the functional spaces~\eqref{energyspace2}, satisfying the energy inequality~\eqref{energy}. As usual, all formal computations can be fully justified by considering smooth solutions of the approximate systems~\eqref{NSMn} instead. The goal of the proof consists in showing the validity of the~$\dot H^s$-bound~\eqref{estimateHsEB12d}, uniformly in~$c$, provided~\eqref{smalldataHs2d} holds initially.

When compared to~\cite{Masmoudi10jmpa}, the main improvement of the present proof comes from the new technology developed in Section~\ref{parabolicregularity}, which allows us to control the flow~$u$ in~$L^2([0,T];L^\infty)$ rather than~$L^1([0,T];L^\infty)$, as was initially done in~\cite{Masmoudi10jmpa}. This temporal refinement will then enable an improved use of Proposition~\ref{estimatewave}, which will result in bounds which are uniform as~$c$ becomes large.

\medskip

We introduce here the following decomposition, which is a very slight variant of the three-dimensional decomposition \eqref{decomposition}:
\begin{equation*}
	u=u_v^\flat+u_v^\sharp+u_e \, ,
\end{equation*}
where~$u_v^\flat$ is the solution of
$$
	\left\{
	\begin{aligned}
		\partial_t u_v^\flat - \mu \Delta u_v^\flat &= 0
		\\
		\operatorname{div} u_v^\flat&= 0
		\\
		u_{v|t=0}^\flat& = u_0
		\, ,
	\end{aligned}
	\right.
$$
and~$u_v^\sharp$ solves
$$
\left\{
	\begin{aligned}
		\partial_t u_v^\sharp - \mu \Delta u_v^\sharp &= -\nabla p_v^\sharp- u \cdot \nabla u  \\
		\operatorname{div}  u_v^\sharp&= 0 \\
		 u_{v|t=0}^\sharp&  = 0\, ,
	\end{aligned}
	\right.
$$
whereas~$u_e$ solves
$$
\left\{
	\begin{aligned}
		\partial_t u_e  - \mu \Delta u_e &= -\nabla p_e  + j \times B \\
		\operatorname{div}  u_e &= 0 \\
		 u_{e|t=0} &  = 0\, .
	\end{aligned}
	\right.
$$

\medskip

The first estimate concerns~$u_v^\flat $.

\begin{lem}\label{estimateuvflat2d}
	There holds that
	\begin{equation*}
		\|u_v^\flat  \|_{L^\infty(\mathbb{R}^+;L^2)}
		\lesssim
		\mathcal{E}_0^\frac 12\, ,
	\end{equation*}
	and
	\begin{equation*}
		\|u_v^\flat  \|_{L^2 ({\mathbb R}^+;L^\infty\cap\dot H^1)}
		\lesssim
		\mathcal{E}_0^\frac 12\, .
	\end{equation*}
\end{lem}

\begin{proof}
	The control of~$u_v^\flat$ in~$L^2L^\infty$ proceeds exactly as the estimate on~$u_1$ in the proof of Corollary~\ref{2dproperty}. Therefore, there only remains to bound the size of~$u_v^\flat$ in~$L^\infty L^2\cap L^2\dot H^1$. In fact, this easily follows from an application of the parabolic regularity estimate~\eqref{parabolicBL}, which yields
	\begin{equation*}
		\|u_v^\flat \|_{L^\infty(\mathbb{R}^+;L^2)\cap L^2 (\mathbb{R}^+;\dot H^1 )}
		\lesssim
		\|u_0\|_{ L^2}\, ,
	\end{equation*}
	thereby completing the justification of the lemma.
\end{proof}

As for~$u_v^\sharp$, we have the following result.

\begin{lem}\label{estimateuvsharp2d}
	There holds that
	\begin{equation*}
		\|u_v^\sharp \|_{L^\infty (\mathbb{R}^+;L^2 )}
		\lesssim \mathcal{E}_0 \, ,
	\end{equation*}
	and
	\begin{equation*}
		\|u_v^\sharp \|_{L^2 (\mathbb{R}^+;L^\infty\cap\dot H^1 )}
		\lesssim
		\|u_v^\sharp \|_{ L^2 ({\mathbb R}^+;\dot B^{1}_{2,1}   )}
		\lesssim \mathcal{E}_0 \, .
	\end{equation*}
\end{lem}

\begin{proof}
	The estimate in $L^2\dot B^{1}_{2,1}$ is obtained by reproducing the control of~$u_2$ from the proof of Corollary~\ref{2dproperty}.
	
	There only remains to control~$u_v^\sharp$ in~$L^\infty L^2$. To this end, we deduce from estimate~\eqref{parabolicBL} and by the Sobolev embedding~$\dot H^\frac 12\subset L^4$ that
	\begin{equation*}
		\begin{aligned}
			\|u_v^\sharp\|_{ L^\infty (\mathbb{R}^+;L^2 )}
			& \lesssim
			\|P(u\cdot\nabla u) \|_{ L^2(\mathbb{R}^+;\dot B^{-1}_{2,2})}
			\lesssim
			\|u\otimes u\|_{ L^2(\mathbb{R}^+;L^2)}
			\\
			& \lesssim
			\|u\|_{ L^4(\mathbb{R}^+;\dot H^\frac 12)}^2
			\lesssim
			\|u\|_{ L^\infty(\mathbb{R}^+;L^2)}
			\|u\|_{ L^2(\mathbb{R}^+;\dot H^1)}\lesssim \mathcal{E}_0\, ,
		\end{aligned}
	\end{equation*}
	which concludes the proof.
\end{proof}

Recall that~$\dot H^1$ barely fails to embed itself continuously into~$L^\infty$, which is a major snag when handling the two-dimensional setting of the Navier--Stokes--Maxwell equations. When dealing with the incompressible Navier--Stokes equations alone, this obstacle is circumvented by exploiting suitable parabolic regularity estimates as shown in Corollary~\ref{2dproperty}. In fact, the ideas of Corollary~\ref{2dproperty} have already been duly exploited in Lemmas~\ref{estimateuvflat2d} and~\ref{estimateuvsharp2d} in the context of the Navier--Stokes--Maxwell system.

However, in order to control the remaining electromagnetic contribution of the flow~$u_e$ in~$L^\infty$, we need now a refined interpolation estimate, which shows that the~$L^\infty$-norm can be controlled by the~$\dot H^1$-norm with some logarithmic help of a higher regularity space. This tame dependence of the~$L^\infty$-norm on higher regularity was crucial in the proof of the main result from~\cite{Masmoudi10jmpa}, whose strategy is closely followed here. We are therefore going to exploit this crucial principle, too. The relevant estimate from~\cite{Masmoudi10jmpa} is recalled in the following lemma. Carefully note that the coming result holds in any dimension and handles high frequencies only. The low frequencies are controlled later on, for convenience.

\begin{lem}\label{high log}
	In any dimension~$d$ and for any~$s>\frac d2$ and~$0\leq t_0<t$, it holds that
	\begin{equation*}
		\|(\operatorname{Id}-S_0)h\|_{L^2([t_0,t];L^\infty)}
		\lesssim \|h\|_{L^2([t_0,t];\dot H^\frac d2)}
		\log^\frac 12 \left(e+\frac{\|h\|_{L^2([t_0,t]];\dot B^s_{2,1})}}{\|h\|_{L^2([t_0,t];\dot H^\frac d2)}}\right) \, .
	\end{equation*}
\end{lem}

\begin{proof}
	In view of the continuous embedding~$\dot B^\frac d2_{2,1}\subset L^\infty$, we only have to bound the norm of~$(\operatorname{Id}-S_0)h$ in~$L^2\dot B^\frac d2_{2,1}$. Thus, we first obtain, for any $N\geq 1$, that
	\begin{equation*}
		\begin{aligned}
			\|(\operatorname{Id}-S_0)h\|_{\dot B^\frac d2_{2,1}}
			&=
			\sum_{k=-1}^{[N]-1}2^{k\frac d2}\|\Delta_k (\operatorname{Id}-S_0)h\|_{L^2}
			+ \sum_{k=[N]}^{\infty}2^{k\frac d2}\|\Delta_k (\operatorname{Id}-S_0)h\|_{L^2}
			\\
			& \lesssim N^\frac 12\|h\|_{\dot H^\frac d2}
			+2^{N(\frac d2-s)}\|h\|_{\dot B^s_{2,1}} \, .
		\end{aligned}
	\end{equation*}
	Hence, integrating in time,
	\begin{equation*}
		\|(\operatorname{Id}-S_0)h\|_{L^2([t_0,t];\dot B^\frac d2_{2,1})}
		\lesssim N^\frac 12\|h\|_{L^2([t_0,t];\dot H^\frac d2)}
		+2^{N(\frac d2-s)}\|h\|_{L^2([t_0,t];\dot B^s_{2,1})} \, .
	\end{equation*}
	Then, following~\cite{Masmoudi10jmpa}, in order to optimize the choice of~$N$, we set
	\begin{equation*}
		N=\frac{1}
		{(s-\frac d2)\log 2}\log\left(2^{s-\frac d2}+\frac{\|h\|_{L^2([t_0,t];\dot B^s_{2,1})}}{\|h\|_{L^2([t_0,t];\dot H^\frac d2)}}\right)\, ,
	\end{equation*}
	which yields
	\begin{equation*}
		\|(\operatorname{Id}-S_0)h\|_{L^2([t_0,t];\dot B^\frac d2_{2,1})}
		\lesssim \|h\|_{L^2([t_0,t];\dot H^\frac d2)}\log^\frac 12\left(e+\frac{\|h\|_{L^2([t_0,t];\dot B^s_{2,1})}}{\|h\|_{L^2([t_0,t];\dot H^\frac d2)}}\right) \, ,
	\end{equation*}
	thus completing the justification of the lemma.
\end{proof}

The low frequencies of the flow will be controlled through an application of the following similar lemma.

\begin{lem}\label{low log}
	In any dimension~$d$ and for any~$0\leq t_0<t$, it holds that
	\begin{equation*}
		\|S_0h\|_{L^2([t_0,t];L^\infty)}
		\lesssim \|h\|_{L^2([t_0,t];\dot H^\frac d2)}\log^\frac 12\left(e+\frac{\|h\|_{L^2([t_0,t];L^2)}}{\|h\|_{L^2([t_0,t];\dot H^\frac d2)}}\right) \, .
	\end{equation*}
\end{lem}

\begin{proof}
	This proof resembles the previous one. In view of the continuous embedding~$\dot B^\frac d2_{2,1}\subset L^\infty$, we only have to bound the norm of~$S_0h$ in~$L^2\dot B^\frac d2_{2,1}$. Thus, we first obtain, for any $N\geq 1$, that
	\begin{equation*}
		\begin{aligned}
			\|S_0h\|_{\dot B^\frac d2_{2,1}}
			&=
			\sum_{k=-[N]}^{0}2^{k\frac d2}\|\Delta_k S_0h\|_{L^2}
			+ \sum_{k=-\infty}^{-([N]+1)}2^{k\frac d2}\|\Delta_k S_0 h\|_{L^2}
			\\
			& \lesssim N^\frac 12\|h\|_{\dot H^\frac d2}
			+2^{-\frac d2 N}\|h\|_{L^2} \, ,
		\end{aligned}
	\end{equation*}
	whence, integrating in time,
	\begin{equation*}
		\|S_0h\|_{L^2([t_0,t];\dot B^\frac d2_{2,1})}
		\lesssim N^\frac 12\|h\|_{L^2([t_0,t];\dot H^\frac d2)}
		+2^{-\frac d2 N}\|h\|_{L^2([t_0,t];L^2)} \, .
	\end{equation*}
	As before, in order to optimize the choice of~$N$, we set
	\begin{equation*}
		N=\frac{1}
		{\frac d2\log 2}\log\left(2^{\frac d2}+\frac{\|h\|_{L^2([t_0,t];L^2)}}{\|h\|_{L^2([t_0,t];\dot H^\frac d2)}}\right)\, ,
	\end{equation*}
	which yields
	\begin{equation*}
		\|S_0h\|_{L^2([t_0,t];\dot B^\frac d2_{2,1})}
		\lesssim \|h\|_{L^2([t_0,t];\dot H^\frac d2)}\log^\frac 12\left(e+\frac{\|h\|_{L^2([t_0,t];L^2)}}{\|h\|_{L^2([t_0,t];\dot H^\frac d2)}}\right) \, .
	\end{equation*}
	The proof of the lemma is thus completed.
\end{proof}

At last, exploiting the preceding interpolation estimates, we control~$u_e$ as follows.

\begin{lem}\label{estimateue2d}
	There holds that
	\begin{equation}\label{basic ue}
		\|u_e\|_{L^\infty(\mathbb{R}^+; L^2)\cap L^2(\mathbb{R}^+; \dot H^1)}
		\lesssim \mathcal{E}_0^\frac 12 + \mathcal{E}_0 \, ,
	\end{equation}
	and, for any~$s\in (0,1)$ and~$0\leq t_0<t$,
	\begin{equation*}
		\begin{aligned}
			\|u_e\|_{L^2([t_0,t];L^\infty)}^2 \hspace{-15mm}&
			\\
			& \lesssim
			\left(\mathcal{E}_0+\mathcal{E}_0^2\right)
			\log\left(e+t-t_0\right)
			+
			\|u_e\|_{L^2([t_0,t];\dot H^1)}^2
			\log\left(e+\frac{\mathcal{E}_0\|B\|_{L^\infty([t_0,t];\dot H^s)}^2}
			{\|u_e\|_{L^2([t_0,t];\dot H^1)}^2}\right)
			\\
			& \lesssim
			\left(\mathcal{E}_0+\mathcal{E}_0^2\right)
			\log\left(e+t-t_0+\frac{\|B\|_{L^\infty([t_0,t];\dot H^s)}^2}
			{1+\mathcal{E}_0}\right)\, .
		\end{aligned}
	\end{equation*}
\end{lem}

\begin{proof}
	First, it is clear from Lemmas~\ref{estimateuvflat2d} and~\ref{estimateuvsharp2d} that
	\begin{equation*}
		\|u_e\|_{L^\infty_t L^2\cap L^2_t  \dot H^1}
		\leq \|u\|_{L^\infty_t  L^2\cap L^2_t  \dot H^1}
		+ \|u_v^\flat\|_{L^\infty_t  L^2\cap L^2_t  \dot H^1}
		+ \|u_v^\sharp\|_{L^\infty_t  L^2\cap L^2_t  \dot H^1}
		\lesssim \mathcal{E}_0^\frac 12 + \mathcal{E}_0 \, .
	\end{equation*}
	Therefore, there only remains to control~$u_e$ in~$L^2 L^\infty$.
	To that end, we deduce from Proposition~\ref{estimateheat} that for any~$s\in (0,1)$ and~$0\leq t_0<t$, 
	\begin{equation*}
		\|u_e\|_{L^2([t_0,t];\dot B^{1+s}_{2,1})}
		\lesssim
		\|P(j\times B)\|_{L^2([t_0,t];\dot B^{-1+s}_{2,1})}
		\lesssim
		\|j\times B\|_{L^2([t_0,t];\dot B^{-1+s}_{2,1})}\, ,
	\end{equation*}
	whence, further employing the paradifferential product rule~\eqref{para3},
	\begin{equation*}
		\|u_e\|_{L^2([t_0,t];\dot B^{1+s}_{2,1})}
		\lesssim
		\big\|\|j\|_{L^2}\|B\|_{\dot H^s}\big\|_{L^2([t_0,t])}
		\lesssim
		\|j\|_{L^2([t_0,t];L^2)}\|B\|_{L^\infty([t_0,t];\dot H^s)}\, .
	\end{equation*}
	Then, combining the preceding estimate with Lemma~\ref{high log}, we find
	\begin{equation*}
		\begin{aligned}
			\|(\operatorname{Id}-S_0)u_e\|_{L^2([t_0,t];L^\infty)}^2
			& \lesssim
			\|u_e\|_{L^2([t_0,t];\dot H^1)}^2
			\log\left(e+\frac{\|u_e\|_{L^2([t_0,t];\dot B^{1+s}_{2,1})}^2}{\|u_e\|_{L^2([t_0,t];\dot H^1)}^2}\right)
			\\
			& \lesssim
			\|u_e\|_{L^2([t_0,t];\dot H^1)}^2
			\log\left(e+\frac{\mathcal{E}_0\|B\|_{L^\infty([t_0,t];\dot H^s)}^2}
			{\|u_e\|_{L^2([t_0,t];\dot H^1)}^2}\right)\, .
		\end{aligned}
	\end{equation*}
	Regarding the low frequencies of~$u_e$, employing Lemma~\ref{low log}, we find
	\begin{equation*}
		\begin{aligned}
			\|S_0u_e\|_{L^2([t_0,t];L^\infty)}^2
			& \lesssim
			\|u_e\|_{L^2([t_0,t];\dot H^1)}^2
			\log\left(e+\frac{\|u_e\|_{L^2([t_0,t];L^2)}^2}{\|u_e\|_{L^2([t_0,t];\dot H^1)}^2}\right)
			\\
			& \lesssim
			\|u_e\|_{L^2([t_0,t];\dot H^1)}^2
			\log\left(e+\frac{(t-t_0)\left(\mathcal{E}_0+\mathcal{E}_0^2\right)}{\|u_e\|_{L^2([t_0,t];\dot H^1)}^2}\right)
			\\
			& \lesssim
			\left(\mathcal{E}_0+\mathcal{E}_0^2\right)
			\log\left(e+t-t_0\right) \, ,
		\end{aligned}
	\end{equation*}
	where we have used~\eqref{basic ue} and the fact that the function~$z\mapsto z\log(e+\frac az)$ on~$z\in\mathbb{R}^+$, for any~$a\geq 0$, is increasing.

	All in all, combining the estimates on high and low frequencies of~$u_e$ gives that
	\begin{equation*}
		\begin{aligned}
			\|u_e\|_{L^2([t_0,t];L^\infty)}^2 \hspace{-15mm}&
			\\
			& \lesssim
			\left(\mathcal{E}_0+\mathcal{E}_0^2\right)
			\log\left(e+t-t_0\right)
			+
			\|u_e\|_{L^2([t_0,t];\dot H^1)}^2
			\log\left(e+\frac{\mathcal{E}_0\|B\|_{L^\infty([t_0,t];\dot H^s)}^2}
			{\|u_e\|_{L^2([t_0,t];\dot H^1)}^2}\right)
			\\
			& \lesssim
			\left(\mathcal{E}_0+\mathcal{E}_0^2\right)
			\log\left(e+t-t_0\right)
			+
			\left(\mathcal{E}_0+\mathcal{E}_0^2\right)
			\log\left(e+\frac{\|B\|_{L^\infty([t_0,t];\dot H^s)}^2}
			{1+\mathcal{E}_0}\right)\, ,
		\end{aligned}
	\end{equation*}
	which concludes the proof of the lemma.
\end{proof}

We may now move on to conclude the proof of Theorem~\ref{mainthm2d}. To this end, observe that Proposition~\ref{estimatewave} (for any~$s\in (0,1)$) combined with Lemmas~\ref{estimateuvflat2d}, \ref{estimateuvsharp2d} and~\ref{estimateue2d} yields that, for any~$0\leq t_0<t$,
\begin{equation*}
	\begin{aligned}
		\|E(t)\|_{\dot H^s}^2+ & \|B(t)\|_{\dot H^s}^2+\sigma\int_{t_0}^t \|cE(\tau)\|_{\dot H^s}^2\, d\tau
		\\
		&
		\leq \left(\|E(t_0)\|_{\dot H^s}^2+\|B(t_0)\|_{\dot H^s}^2\right)
		\exp \left(C_1 \int_{t_0}^t \|u(\tau)\|_{L^\infty\cap \dot H^1}^2\, d\tau\right)
		\\
		& \leq
		\left(\|E(t_0)\|_{\dot H^s}^2+\|B(t_0)\|_{\dot H^s}^2\right)
		\left(e+t-t_0\right)^{C_2
		\left(\mathcal{E}_0+\mathcal{E}_0^2\right)}
		\\
		& \quad\times
		\left(e+\frac{\mathcal{E}_0\|B\|_{L^\infty([t_0,t];\dot H^s)}^2}
		{\|u_e\|_{L^2([t_0,t];\dot H^1)}^2}\right)^{C_2\|u_e\|_{L^2([t_0,t];\dot H^1)}^2}
		\, ,
	\end{aligned}
\end{equation*}
for some constants~$C_1,C_2>0$ depending only on fixed parameters.

Using that the function~$z\mapsto (e+\frac az)^z$ on~$z\in\mathbb{R}^+$, for any~$a\geq 0$, is increasing, and defining, for all~$0\leq t_0<t$,
\begin{equation*}
	\begin{aligned}
		G(t_0,t)&:=
		\sup_{r\in [t_0,t]}
		\left(\|E(r)\|_{\dot H^s}^2+\|B(r)\|_{\dot H^s}^2\right)
		\\
		G(t_0,t_0)&:= \|E(t_0)\|_{\dot H^s}^2+\|B(t_0)\|_{\dot H^s}^2\, ,
	\end{aligned}
\end{equation*}
we deduce that
\begin{equation}\label{2d1}
	G(t_0,t)
	\leq
	G(t_0,t_0)
	\left(e+t-t_0\right)^{C_2
	\left(\mathcal{E}_0+\mathcal{E}_0^2\right)}
	\left(e+\frac{\mathcal{E}_0G(t_0,t)}
	{\|u_e\|_{L^2([t_0,t];\dot H^1)}^2}\right)^{C_2\|u_e\|_{L^2([t_0,t];\dot H^1)}^2}
	\, .
\end{equation}
Recall that all unkowns are assumed to be smooth, for all estimates are to be performed on the regularized system~\eqref{NSMn}. In particular, $G(t_0,t)$ is assumed here to be continuous.

The proof of the theorem will be complete upon showing that~\eqref{2d1} entails the global bound
\begin{equation}\label{2d2}
	\mathcal{E}_0G(0,t)
	\leq
	\left(e+\mathcal{E}_0G(0,0)+\frac{t}{1+\mathcal{E}_0+\mathcal{E}_0^2}\right)^{C_* 2^{C_*(\mathcal{E}_0+\mathcal{E}_0^2)}}
	\, ,
\end{equation}
for some possibly large constant~$C_*>0$ only depending on fixed parameters.

In order to establish the validity of~\eqref{2d2}, using that~$\|u_e\|_{L^2(\mathbb{R}^+;\dot H^1)}$ is finite by virtue of~\eqref{basic ue}, we first consider a partition
\begin{equation*}
	0=t_0<t_1<t_2<\ldots<t_n<t_{n+1}=\infty \, ,
\end{equation*}
for some~$n\in\mathbb{N}$, such that, for each~$i=1,\ldots,n$,
\begin{equation*}
	C_2\|u_e\|_{L^2([t_{i-1},t_i];\dot H^1)}^2 = \frac 12
	\quad\text{and}\quad
	C_2\|u_e\|_{L^2([t_n,\infty);\dot H^1)}^2\leq \frac 12 \, \cdotp
\end{equation*}
In particular, it holds that
\begin{equation*}
	\frac i2
	= C_2\|u_e\|_{L^2([0,t_i];\dot H^1)}^2
	\leq C_2\|u_e\|_{L^2([0,t];\dot H^1)}^2 \leq \frac {i+1}2 \, ,
\end{equation*}
for every~$t\in [t_{i},t_{i+1})$, with~$i=0,\ldots,n$.

It then follows from~\eqref{2d1} that, for each~$i=0,\ldots,n$ and all~$t\in[t_i,t_{i+1})$,
\begin{equation*}
	G(t_i,t)
	\leq
	G(t_i,t_i)
	\left(e+t-t_i\right)^{C_2
	\left(\mathcal{E}_0+\mathcal{E}_0^2\right)}
	\left(e+2\mathcal{E}_0G(t_i,t)\right)^\frac 12
	\, ,
\end{equation*}
which implies the weaker inequality
\begin{equation*}
	e+2\mathcal{E}_0G(t_i,t)
	\leq
	(e+2\mathcal{E}_0G(t_i,t_i))
	\left(e+t-t_i\right)^{C_2
	\left(\mathcal{E}_0+\mathcal{E}_0^2\right)}
	\left(e+2\mathcal{E}_0G(t_i,t)\right)^\frac 12
	\, ,
\end{equation*}
and therefore
\begin{equation}\label{G1}
	e+2\mathcal{E}_0G(t_i,t)
	\leq
	(e+2\mathcal{E}_0G(t_i,t_i))^2
	\left(e+t-t_i\right)^{2C_2
	\left(\mathcal{E}_0+\mathcal{E}_0^2\right)}
	\, .
\end{equation}

Next, observe that~$G(t_i,t_i)\leq G(t_{i-1},t_i)$, for every~$i=1,\ldots,n$. Thus, given any~$t\in [t_k,t_{k+1})$, for some~$k\in\left\{0,1,\ldots,n\right\}$, applying recursively the bound~\eqref{G1}, we obtain that
\begin{equation*}
	\begin{aligned}
		\frac{e+2\mathcal{E}_0G(t_k,t)}{\left(e+t-t_k\right)^{2C_2
		\left(\mathcal{E}_0+\mathcal{E}_0^2\right)}} \hspace{-10mm} &
		\\
		& \leq
		(e+2\mathcal{E}_0G(t_{k-1},t_k))^2
		\\
		& \leq
		(e+2\mathcal{E}_0G(t_{k-2},t_{k-1}))^4
		\left(e+t_{k}-t_{k-1}\right)^{4C_2
		\left(\mathcal{E}_0+\mathcal{E}_0^2\right)}
		\\
		& \leq \ldots
		\\
		& \leq
		(e+2\mathcal{E}_0G(t_{0},t_{1}))^{2^k}
		\prod_{j=2}^k\left(e+t_{k+2-j}-t_{k+1-j}\right)^{2^jC_2\left(\mathcal{E}_0+\mathcal{E}_0^2\right)}
		\\
		& \leq
		(e+2\mathcal{E}_0G(t_{0},t_{0}))^{2^{k+1}}
		\prod_{j=2}^{k+1}\left(e+t_{k+2-j}-t_{k+1-j}\right)^{2^jC_2\left(\mathcal{E}_0+\mathcal{E}_0^2\right)}
		\, .
	\end{aligned}
\end{equation*}
Now, employing that the arithmetic mean is always larger than the geometric mean, we see that
\begin{equation*}
	\left(e+t-t_k\right)
	\prod_{j=2}^{k+1}\left(e+t_{k+2-j}-t_{k+1-j}\right)
	\leq
	\left(e+\frac{t}{k+1}\right)^{k+1} \, .
\end{equation*}
Therefore, we deduce that
\begin{equation*}
	e+2\mathcal{E}_0G(t_k,t)
	\leq
	(e+2\mathcal{E}_0G(t_{0},t_{0}))^{2^{k+1}}
	\left(e+\frac{t}{k+1}\right)^{(k+1)2^{k+1}C_2\left(\mathcal{E}_0+\mathcal{E}_0^2\right)}
	\, ,
\end{equation*}
for every~$k\in\{0,1,\ldots,n\}$ and~$t\in[t_k,t_{k+1})$.

Further employing estimate~\eqref{basic ue} combined with the fact that~$n\lesssim \|u_e\|_{L^2(\mathbb{R}^+;\dot H^1)}^2$, and using that~$z\mapsto (e+\frac az)^z$ is increasing, for any~$a\geq 0$, we obtain
\begin{equation*}
	\mathcal{E}_0G(t_k,t)
	\leq
	\left(e+\mathcal{E}_0G(t_{0},t_{0})+\frac{t}{1+\mathcal{E}_0+\mathcal{E}_0^2}\right)^{C_* 2^{C_*(\mathcal{E}_0+\mathcal{E}_0^2)}}
	\, ,
\end{equation*}
for every~$k\in\{0,1,\ldots,n\}$ and~$t\in[t_k,t_{k+1})$, for some possibly large constant~$C_*>0$ only depending on fixed parameters. At last, since
\begin{equation*}
	G(0,t)=\max\left\{G(t_0,t_1),G(t_1,t_2),\ldots,G(t_{k-1},t_k),G(t_k,t)\right\},
\end{equation*}
it is readily seen that~\eqref{2d2} holds for every~$t\geq 0$, which concludes the proof of the theorem. \qed

\appendix


\section{Littlewood--Paley decompositions and Besov spaces}\label{LP decomposition}

We denote the Fourier transform
\begin{equation*}
	\hat f(\xi):=\mathcal{F}f\left(\xi\right):=\int_{\mathbb{R}^d} e^{- i \xi \cdot x} f(x) dx\,,
\end{equation*}
and its inverse
\begin{equation*}
	\tilde g(x):=\mathcal{F}^{-1} g\left(x\right):=\frac{1}{\left(2\pi\right)^d}\int_{\mathbb{R}^d} e^{i x \cdot \xi} g(\xi) d\xi\,.
\end{equation*}
We introduce now a standard Littlewood-Paley decomposition of the frequency space into dyadic blocks. To this end, let $\psi(\xi),\varphi(\xi)\in C_c^\infty\left(\mathbb{R}^d\right)$ be such that
\begin{equation*}
	\begin{gathered}
		\psi,\varphi\geq 0\text{ are radial},
		\quad \mathrm{supp}\, \psi\subset\left\{|\xi|\leq 1\right\},
		\quad\mathrm{supp}\, \varphi\subset\left\{\frac{1}{2}\leq |\xi|\leq 2\right\}\\
		\text{and}\quad1= \psi(\xi)+\sum_{k=0}^\infty \varphi\left(2^{-k}\xi\right),\quad\text{for all }\xi\in\mathbb{R}^d.
	\end{gathered}
\end{equation*}
Defining the scaled functions $\displaystyle \psi_{k}(\xi):=\psi\left(2^{-k}\xi\right)$ and $\displaystyle\varphi_{k}(\xi):=\varphi\left(2^{-k}\xi\right)$,      one has then  
\begin{equation*}
	\begin{gathered}
		\mathrm{supp}\,\psi_{k}\subset\left\{ |\xi|\leq 2^k\right\},
		\quad \mathrm{supp}\,\varphi_{k}\subset\left\{ 2^{k-1}\leq |\xi|\leq 2^{k+1}\right\}\\
		\text{and}\quad 1\equiv \psi+\sum_{k=0}^\infty \varphi_{k} \, .
	\end{gathered}
\end{equation*}
Notice that outside $0$ one also has
$$
1\equiv \sum_{k=-\infty}^\infty \varphi_{k} \, .
$$
Furthermore, we shall use the Fourier multiplier operators
\begin{equation*}
	S_k ,\Delta_k:\mathcal{S}'\left(\mathbb{R}^d\right)\rightarrow\mathcal{S}'\left(\mathbb{R}^d\right)
\end{equation*}
(here $\mathcal{S}'$ denotes the space of tempered distributions) defined by
\begin{equation*}
	S_k f:=\mathcal{F}^{-1}\psi_k\mathcal{F}f= \left(\mathcal{F}^{-1}\psi_k\right)*f
	\quad\text{and}\quad
	\Delta_k f:=\mathcal{F}^{-1}\varphi_k\mathcal{F}f= \left(\mathcal{F}^{-1}\varphi_k\right)*f \, ,
\end{equation*}
so that
\begin{equation*}
	S_0 f+\sum_{k=0}^\infty\Delta_{ k}f=f \, ,
\end{equation*}
where the series is convergent in $\mathcal{S}'$.
Similarly one has
 \begin{equation*}
 \sum_{k=-\infty}^\infty\Delta_{ k}f=f \, ,
\end{equation*}
in~$\mathcal{S}'$, provided
\begin{equation}\label{origin}
	\lim_{k\to -\infty}\|S_kf\|_{L^\infty}=0 \, .
\end{equation}
Observe that~\eqref{origin} holds as soon as~$\hat f$ is locally integrable around the origin or~$S_0f$ belongs to~$ L^p(\mathbb{R}^d)$, for some~$1\leq p<\infty$. In particular, note that the above property excludes non-zero polynomials.

\medskip

Now, we define the homogeneous Besov space $\dot B^{s}_{p,q}\left(\mathbb{R}^d\right)$, for any~$s \in \mathbb{R}$ and~$1\leq p,q\leq \infty$, as the subspace of tempered distributions satisfying~\eqref{origin} endowed with the norm
\begin{equation*}
	\left\|f\right\|_{\dot B^{s}_{p,q}\left(\mathbb{R}^d\right)}=
	\left(
	\sum_{k=-\infty}^\infty 2^{ksq}
	\left\|\Delta_{k}f\right\|_{L^p\left(\mathbb{R}^d\right)}^q\right)^\frac{1}{q}\, ,
\end{equation*}
if~$q<\infty$, and with the obvious modifications in case~$q=\infty$. It holds that~$\dot B^s_{p,q}$ is a Banach spaces if~$s<\frac dp$, or if~$s=\frac dp$ and~$q=1$ (see~\cite[Theorem 2.25]{bahouri}).

We also introduce the homogeneous Sobolev space~$\dot H^s\left(\mathbb{R}^d\right)$, for any~$s\in\mathbb{R}$, as the subspace of tempered distributions whose Fourier transform is locally integrable endowed with the norm
\begin{equation*}
	\left\|f\right\|_{\dot H^s}=\left(\int_{\mathbb{R}^d}|\xi|^{2s}|\hat f(\xi)|^2 d\xi\right)^\frac 12\,.
\end{equation*}
It holds that~$\dot H^s$ is a Hilbert space if~$s<\frac d2$ (see~\cite[Proposition 1.34]{bahouri})

Since any tempered distribution whose Fourier transform is locally integrable automatically satisfies~\eqref{origin}, it is clear that~$\dot H^s\subset \dot B^s_{2,2}$. Conversely, suppose that~$s<\frac d2$ and consider any~$f\in \dot B^s_{2,2}$. Then, each~$\Delta_k f$ belongs to~$L^2$ and~$\hat f$ is therefore locally integrable away from the origin. But~$\hat f$ is also integrable near the origin, for
\begin{equation*}
	\begin{aligned}
		\left\|\psi_0 \hat  f \right\|_{L^1}
		=
		\left\|\sum_{k\leq -1}\mathcal{F}(\Delta_k f)\right\|_{L^1}
		& \leq
		\sum_{k\leq -1}
		\left\|\mathds{1}_{\left\{2^{k-1}\leq |\xi|\leq 2^{k+1}\right\}}\mathcal{F}(\Delta_k f)\right\|_{L^1}
		\\
		& \lesssim
		\sum_{k\leq -1}2^{k\frac d2}
		\left\|\mathcal{F}(\Delta_k f)\right\|_{L^2}
		\lesssim \|f\|_{\dot B^s_{2,2}}\, ,
	\end{aligned}
\end{equation*}
which implies that~$\dot H^s= \dot B^s_{2,2}$ whenever~$s<\frac d2$.

\medskip

Now, we recall two important product rules of paradifferential calculus in homogeneous Besov spaces. Both rules can be deduced directly from Theorems~2.47 and~2.52 in~\cite[Section 2.6]{bahouri}.

First, for any~$-\frac d2<s,t<\frac d2$ with~$s+t>0$, we have that
\begin{equation*}
	\|fg\|_{\dot B^{s+t-\frac d2}_{2,1}} \lesssim \|f\|_{\dot H^s}\|g\|_{\dot H^t}\, ,
\end{equation*}
for all~$f\in\dot H^s$ and~$g\in\dot H^t$.

Second, for all~$-\frac d2<s<\frac d2$, it holds that
\begin{equation*}
	\|fg\|_{\dot H^s} \lesssim \|f\|_{L^\infty\cap \dot B^\frac d2_{2,\infty}}\|g\|_{\dot H^s}\, ,
\end{equation*}
for all~$f\in L^\infty\cap \dot B^\frac d2_{2,\infty}$ and~$g\in\dot H^s$. In particular, further employing the continuous injection~$\dot B^\frac d2_{2,1}\subset L^\infty\cap \dot B^\frac d2_{2,\infty}$, observe that
\begin{equation*}
	\|fg\|_{\dot H^s} \lesssim \|f\|_{\dot B^\frac d2_{2,1}}\|g\|_{\dot H^s}\, ,
\end{equation*}
for all~$f\in \dot B^\frac d2_{2,1}$ and~$g\in\dot H^s$.

These product rules are used several times throughout this work.

\medskip

Finally, recall that, for any~$T>0$, $s \in \mathbb{R}$ and~$1\leq p,q,r\leq \infty$, with~$s<\frac dp$ (or~$s=\frac dp$ and~$q=1$), the spaces~$ L^r\left( (0,T) ; \dot B^{s}_{p,q}\left(\mathbb{R}^d\right) \right)$ are defined as~$L^r$-spaces with values in the Banach spaces~$\dot B^{s}_{p,q}$. In addition to these vector-valued Lebesgue spaces, we further define the spaces~$ \widetilde L^r\left( (0,T) ; \dot B^{s}_{p,q}\left(\mathbb{R}^d\right) \right)$ as the subspaces of tempered distributions such that
\begin{equation*}
	\lim_{k\to -\infty}\|S_kf\|_{L^r\left((0,T);L^p\left(\mathbb{R}^d\right)\right)}=0 \, ,
\end{equation*}
endowed with the norm
\begin{equation*}
	\left\|f\right\|_{ \widetilde L^r \left( (0,T) ; B^{s}_{p,q}\left(\mathbb{R}^d\right) \right)}
	=
	\left( \sum_{k=-\infty}^\infty 2^{ksq}
	\left\|\Delta_{k}f\right\|_{L^r\left((0,T);L^p\left(\mathbb{R}^d\right)\right) }^q\right)^\frac{1}{q}\, ,
\end{equation*}
if~$q<\infty$, and with the obvious modifications in case~$q=\infty$. This kind of spaces was first introduced by Chemin and Lerner in~\cite{chemin2} and has been used in a large variety of problems since then.

One can easily check that, if~$q \geq r $, then
\begin{equation*}
	 L^r\left( (0,T) ; \dot B^{s}_{p,q}\left(\mathbb{R}^d\right) \right)
	\subset  \widetilde L^r\left( (0,T) ; \dot B^{s}_{p,q}\left(\mathbb{R}^d\right) \right)\, ,
\end{equation*}
and that, if~$q \leq r $, then
\begin{equation*}
	\widetilde L^r\left( (0,T) ; \dot B^{s}_{p,q}\left(\mathbb{R}^d\right)  \right)
	\subset L^r\left( (0,T); \dot B^{s}_{p,q}\left(\mathbb{R}^d\right) \right)\, .
\end{equation*}
We refer the reader to~\cite[Section 2.6.3]{bahouri} for more details on Chemin--Lerner spaces.

	\bibliographystyle{plain}
	\bibliography{NS_Maxwell}

\end{document}